\documentclass[twoside,11pt]{article}

\usepackage[abbrv,preprint]{jmlr2e}
\usepackage{microtype}
\usepackage{graphicx}
\usepackage{subfigure}
\usepackage{booktabs} 

\usepackage{amsmath, amsfonts,amssymb,verbatim}

\usepackage{hyperref}
\hypersetup{draft}

\usepackage{algorithm}
\usepackage{algorithmicx}
\usepackage{algpseudocode}
\usepackage[export]{adjustbox}
\def\R{\mathbb{R}}

\newcommand{\1}{\boldsymbol{1}}

\DeclareMathOperator*{\argmax}{arg\,max}

\ShortHeadings{Training of Binarized Deep Neural Networks}{Bah and Kurtz}
\firstpageno{1}

\begin{document}
\title{Efficient and Robust Mixed-Integer Optimization Methods for Training Binarized Deep Neural Networks}

\author{\name Jannis Kurtz \email jannis.kurtz@uni-siegen.de \\
       \addr University of Siegen, School of Economic Disciplines, 57068 Siegen, Germany\\
       \AND
       \name Bubacarr Bah \email bubacarr@aims.ac.za \\
       \addr African Institute for Mathematical Sciences, Cape Town 7945, South Africa
       }

\editor{Kevin Murphy and Bernhard Sch{\"o}lkopf}

\maketitle

\begin{abstract}
Compared to classical deep neural networks its binarized versions can be useful for applications on resource-limited devices due to their reduction in memory consumption and computational demands. In this work we study deep neural networks with binary activation functions and continuous or integer weights (BDNN). We show that the BDNN can be reformulated as a mixed-integer linear program with bounded weight space which can be solved to global optimality by classical mixed-integer programming solvers. Additionally, a local search heuristic is presented to calculate locally optimal networks. Furthermore to improve efficiency we present an iterative data-splitting heuristic which iteratively splits the training set into smaller subsets by using the $k$-mean method. Afterwards all data points in a given subset are forced to follow the same activation pattern, which leads to a much smaller number of integer variables in the mixed-integer programming formulation and therefore to computational improvements. Finally for the first time a robust model is presented which enforces robustness of the BDNN during training. All methods are tested on random and real datasets and our results indicate that all models can often compete with or even outperform classical DNNs on small network architectures confirming the viability for applications having restricted memory or computing power.
\end{abstract}


\section{Introduction}
Deep learning (DL) methods have of-late reinvigorated interest in artificial intelligence and data science, and they have had many successful applications in computer vision, natural language processing, and data analytics \citep{lecun2015deep}. The training of deep neural networks relies mostly on (stochastic) gradient descent, hence the use of differentiable activation functions like ReLU, sigmoid or the hyperbolic tangent is the state-of-the-art \citep{rumelhart1986learning,goodfellow2016deep}. On the contrary binary activations, which may be more analogous to biological activations, present a training challenge due to non-differentiability and even discontinuity. If additionally the weights are considered to be binary, the use of binary activation networks significantly reduces the computation and storage complexities, provides for better interpretation of solutions, and has the potential to be more robust to adversarial perturbations than the continuous networks \citep{qin2020binary}. Furthermore low-powered computations may benefit from binarized networks as a form of coarse quantization \citep{plagianakos2001training,bengio2013estimating,courbariaux2015binaryconnect,rastegari2016xnor}. Moreover, gradient descent-based training behaves like a black box, raising a lot of questions regarding the explainability and interpretability of internal representations \citep{hampson1990representing,plagianakos2001training,bengio2013estimating}.  

The interest in BDNNs goes back to \citet{mcculloch1943logical} where BDNNs were used to simulate Boolean functions. However, until the beginning of this century concerted efforts were made to train these networks either by specific schemes \citep{gray1992training,kohut2004boolean} or via back propagation by modifications of the gradient descent method \citep{widrow1988neural,toms1990training,barlett1992using,goodman1994learning,corwin1994iterative,plagianakos2001training}. More recent work regarding the back propagation, mostly motivated by the low complexity of computation and storage, build-on the pioneering works of \citet{bengio2013estimating,courbariaux2015binaryconnect,hubara2016binarized,rastegari2016xnor,kim2016bitwise}; see \cite{qin2020binary} for a detailed survey. Regarding the generalization error of BDNN, it was already proved that the VC-Dimension of deep neural networks with binary activation functions is $\rho\log (\rho)$, where $\rho$ is the number of weights of the BDNN; see \citet{baum1989size,maass1994perspectives,sakurai1993tighter}.

On the other hand, integer programming (IP) is known as a powerful tool to model a huge class of real-world optimization problems \citep{wolsey1998integer}. Recently it was successfully applied to evaluate trained deep neural networks \citep{fischetti2018deep,tjeng2017evaluating, anderson2020strong}. In \citet{lazarus2021mixed,jia2020efficient} efficient methods for verification of BDNNs were derived. 

Integer programming models for the training of BDNNs benefit from their high flexibility, since new constraints or regularizers can be added easily to the IP model without changing the solution methods. On the other hand, they lead to better interpretability due to the well understood polyhedral geometry and further mixed-integer programming theory. Despite the huge success in development of integer programming solvers like CPLEX or Gurobi, integer programming models for BDNNs often suffer under high computational demands during the training process. In \citet{icarte2019training} BDNNs with weights restricted to $\{-1,0,1\}$ are trained by a hybrid method based on constraint programming and mixed-integer programming. In \citet{thorbjarnarson2020training} mixed-integer-programming formulations are used to train BDNNs with different loss functions and an ensemble method is derived. The derived methods are compared to gradient-based methods. In \citet{khalil2018combinatorial} the authors calculate optimal adversarial examples for BDNNs using a MIP formulation and integer propagation. Furthermore, robust optimization approaches were used to protect against adversarial attacks for other machine learning methods \citep{xu2009robustness,xu2009robust,bertsimas2019robust}. Some of the results of this work were already presented by the authors in the ICML workshop paper \cite{bah2020integer}.

\paragraph{Contributions:} In this manuscript, we consider classification problems and show that the BDNN can be trained via a MIP formulation where the weight space can be assumed to be bounded if the weights are chosen to be continuous. The latter problem can be solved to global optimality by classical MIP solvers. All results can also be applied to regression problems and BDNNs with integer weights. We note that it is straight forward to extend the IP formulation to more general settings and several variations of the model. The key contributions of this work are the following.
\begin{itemize}
	\item The introduction of two implementation strategies that speed up the BDNN training, i.e. a local search and a data-splitting algorithm. While the local search algorithm was derived to circumvent the quadratic structure of the MIP formulation, the data-splitting algorithm improves efficiency by reducing the number of integer variables.
	\item The proposition of the first BDNN model which incorporates a robust optimization method during training, leading to BDNNs which are robust against data uncertainty. While other approaches are either able to verify robustness after training or incorporate adversarial examples during training without achieving a robustness guarantee, our method is able to train BDNNs with a given robustness guarantee.
	\item Simulations that corroborate our theoretical findings but also give new insights into the trade-offs resulting from the BDNN. We tested all presented methods on random and real datasets and compare the BDNN to a deep neural network using ReLU activations (DNN). Despite scalability issues and a slightly worse accuracy on random datasets, the results indicate that the heuristic version outperforms the DNN on the \textit{Breast Cancer Wisconsin} dataset. On the other hand the iterative data-splitting method turns out to be the most efficient method for training the BDNN leading to high accuracies on small network architectures. The robust BDNN model shows that the well-known trade-off between robustness and accuracy does not hold for BDNNs, leading to better or worse accuracies for different attack and defense levels.  
\end{itemize}  

\paragraph{Organization of the paper:}
The rest of the paper is organized as follows. In Section \ref{sec:DNN} we present the theoretical framework of the BDNN with the MIP formulation in Section \ref{sec:BDNN}. In Section \ref{sec:heuristic} we present a local search heuristic to calculate locally optimal networks. Additionally we present an iterative data-splitting algorithm in Section \ref{sec:data_splitting}. In Section \ref{sec:adversarialAttacks} we propose an approach for robustifying BDNNs and finally in Section \ref{sec:computations} we present the results of our numerical experiments.

\section{Binarized and Mixed-Binarized Neural Networks}\label{sec:DNN}
In this section we reformulate the BDNN as a mixed-integer program in Section \ref{sec:BDNN} and propose heuristic solution methods using a local search algorithm in Section \ref{sec:heuristic} and an iterative data splitting method in Section \ref{sec:data_splitting}.

\subsection{Mixed-Integer Programming Formulation}\label{sec:BDNN}
In this work we study a generalization of \textit{binarized deep neural networks} (BDNN), i.e. classical deep neural networks with binary activation functions where the weights are restricted either to a convex or discrete set. As in the classical framework, for a given input vector $\mathop{x\in \R^n}$ we study classification functions $f$ of the form 
\[\mathop{f(x)=\sigma^K\left( W^K\sigma^{K-1}\left( W^{K-1}\ldots ,\sigma^1\left(W^1x\right)\ldots \right)\right)}\]
for weight matrices $W^k\in D_k\subset \R^{d_{k}\times d_{k-1}}$ and activation functions $\sigma^k$, which are applied component-wise. The dimension $d_{k}$ is called the \textit{width} of the \textit{$k$-th layer}. All our results can be applied to arbitrary convex or discrete sets $D_k$, but in the following we focus on the two cases where $D_k=[-1,1]^{d_{k}\times d_{k-1}}$ or $D_k=\{-1,0,1\}^{d_{k}\times d_{k-1}}$. Furthermore bias vectors $b^k\in\R^{d_k}$ can be easily incorporated into each layer $k$ which are omitted for ease of notation. In contrast to the recent developments of the field we consider the activation functions to be binary, more precisely each function is of the form
\begin{equation}\label{eq:defActivationFunction}
\sigma^k(\alpha ) = \begin{cases}
0 & \text{ if } \alpha < \lambda_k \\
1 & \text{ otherwise}
\end{cases} 
\end{equation}

for $\alpha\in \R$ where the parameters $\lambda_k\in \R$ can be learned by our model simultaneously with the weight matrices which is normally not the case in the classical neural network approaches. Note that it is also possible to fix the values $\lambda_k$ in advance.

In the following we use the notation $[p]:=\left\{ 1,\ldots ,p\right\}$ for $p\in \mathbb N$. Given a set of labeled training samples \[X\times Y = \left\{ (x^i,y^i) \ | \ i\in [m]\right\}\subset \R^n\times \left\{ 0,1\right\}\] we consider loss functions
\begin{equation}\label{eq:defLossFunction}
\ell: \left\{ 0,1 \right\} \times \R^{d_K} \to \R
\end{equation}
and the task is to find the optimal weight matrices which minimize the empirical loss over the training samples, i.e. we want to solve the problem
\begin{equation}\label{eq:DNNDefinition}
\begin{aligned}
\min \ & \sum_{i=1}^{m} \ell\left( y^i , z^i\right) \\
s.t. \quad & z^i = \sigma^K\left( W^K\sigma^{K-1}\left(\ldots \sigma^1\left(W^1x^i\right)\ldots \right)\right) \ \ \forall i\in [m] \\
& W^k\in D_k \ \ \forall k\in[K] \\
& \lambda_k\in \R \ \ \forall k\in [K]
\end{aligned}
\end{equation}
for given dimensions $d_0,\ldots d_{K}$ where $d_0=n$. We set $d_{K}=2$, which is the number of classes, and use labels $y\in\left\{ 0,1\right\}$ indicating the class of the corresponding data point. We minimize the empirical classification error, i.e. we apply the loss function
\begin{equation}\label{eq:empiricalLoss}
\ell\left( y , z\right) = (2y-1)z_1 + (1-2y)z_2.
\end{equation}
Note that for class label $y=0$ the minimal loss is given by $z=(1,0)$, while for $y=1$ the minimal loss is attained at $z=(0,1)$. After the training process the predicted class of a data point $x$ is $$\argmax_{i=1,2} ~z_i $$
where $z\in \R^2$ is the output of the BDNN after applying the data point $x$. The problem can easily be extended to multiclass classification tasks by setting $d_K=c$, where $c$ is the number of classes and adjusting the loss function appropriately.

All results in this work also hold for \textit{regression problems}, more precisely for loss functions
\[
\ell_r: \R \times \R \to \R
\]
where we minimize the empirical loss $\sum_{i=1}^{m} \ell_r\left( y^i , z^i\right)$ instead. This case can be modeled by choosing the activation function of the last layer $\sigma^K$ as the identity and $d_K=1$. Any classical loss functions may be considered, e.g. the squared loss $\ell (y, z ) = \| y - z\|_2^2$.

%
%

In the following lemma we show how to reformulate Problem \eqref{eq:DNNDefinition} as a mixed-integer program with bounded $D_k$.
\begin{lemma}\label{lem:MINLP_formulation}
Assume the Euclidean norm of each data point in $X$ is bounded by $r$ and $D_k=\R^{d_k\times d_{k-1}}$ for each $k\in [K]$, then Problem \ref{eq:DNNDefinition} is equivalent to the mixed-integer non-linear program
\begin{align}
& \min \ ~\sum_{i=1}^{m} \ell\left( y^i , u^{i,K}\right) \quad s.t. \ \  \label{eq:MINLP_objective}\\
& W^1 x^i < M_1 u^{i,1} + \1\lambda_1 \quad i\in [m]  \label{eq:MINLP_firstlayer}\\
& W^1 x^i \ge M_1 (u^{i,1}-1) + \1\lambda_1 \quad i\in [m] \label{eq:MINLP_firstlayer2}\\
& W^k u^{i,k-1} < M_k u^{i,k} + \1\lambda_k  \ \ \forall k\in [K]\setminus \{1\}, i\in [m]\label{eq:MINLP_layers}\\
& W^k u^{i,k-1} \ge M_k(u^{i,k}-1) + \1\lambda_k  \ \ \forall k\in [K]\setminus \{1\}, i\in [m]\label{eq:MINLP_layers2}\\
& W^k \in[-1,1]^{d_k\times d_{k-1}}, \ \ \lambda_k\in[-1,1] \ \ \forall k\in [K] \label{eq:MINLP_variablesW}\\ 
& u^{i,k}\in \left\{ 0,1\right\}^{d_k} \ \ \forall i\in [m], k\in [K],\label{eq:MINLP_variables}
\end{align}
where $M_1:=(nr+1)$, $M_k:=(d_{k-1}+1)$ and $\1$ is the all ones vector.
\end{lemma}

\begin{proof}
First we show that, due to the binary activation functions, we may assume $W^k\in[-1,1]^{d_k\times d_{k-1}}$ and $\lambda_k\in [-1,1]$ for all $k\in [K]$ in our model. To prove this assume we have any given solution $W^1,\ldots , W^K$ and corresponding $\lambda_1,\ldots ,\lambda_K$ of problem \eqref{eq:DNNDefinition} with arbitrary values in $\R$. Consider any fixed layer $k\in [K]$. The $k$-th layer receives a vector $h^{k-1}\in\{ 0, 1\}^{d_{k-1}}$ from the previous layer, which is applied to $W^k$ and afterwards the activation function $\sigma^k$ is applied component-wise, i.e. the output of the $k$-th layer is a vector
\[
h^k_j = \begin{cases} 0 & \text{ if } (w_j^k)^\top h^{k-1}<\lambda_k \\ 1 & \text{ otherwise}\end{cases}
\]
where $w_j^k$ is the $j$-th row of the matrix $W^k$. Set 
\[\beta:=\max\{|\lambda_k|,\max_{\substack{j=1,\ldots ,d_k\\ l=1,\ldots ,d_{k-1}}}|w_{jl}^k|\}\]
and define $\tilde W^k:=\frac{1}{\beta} W^k$ and $\tilde \lambda_k:=\frac{1}{\beta}\lambda_k$. Then replacing $W^k$ by $\tilde W^k$ and $\lambda_k$ by $\tilde \lambda_k$ in the $k$-th layer yields the same output vector $h^k$, since the inequality $ (w_j^k)^\top h^{k-1}<\lambda_k$ holds if and only if the inequality $ (\tilde w_j^k)^\top h^{k-1}<\tilde \lambda_k$ holds. Furthermore all entries of $\tilde W^k$ and $\tilde \lambda_k$ have values in $[-1,1]$.

Next we show that the constraints \eqref{eq:MINLP_firstlayer}--\eqref{eq:MINLP_layers2} correctly model the equation 
\begin{align}
z^i	:= & ~u^{i,K} \nonumber\\
	= & ~\sigma^K\left( W^K\sigma^{K-1}\left( W^{K-1}\ldots ,\sigma^1\left(W^1x^i\right)\ldots \right)\right) \nonumber
\end{align}
of Problem \ref{eq:DNNDefinition}. The main idea is that the $u^{i,k}$-variables model the output of the activation functions of data point $i$ in layer $k$, i.e. they have value $0$ if the activation value is $0$ or value $1$ otherwise. More precisely for any solution $W^1, \ldots, W^k$ and $\lambda_1,\ldots ,\lambda_k$ of the Problem in Lemma \ref{lem:MINLP_formulation} the variable $u_j^{i,1}$ is equal to $1$ if and only if $(w_j^1)^\top x^i\ge \lambda_1$ since otherwise Constraint \ref{eq:MINLP_firstlayer2} would be violated. Note that if $u_j^{i,1}=1$, then Constraint \ref{eq:MINLP_firstlayer} is always satisfied since all entries of $W^1$ are in $[-1,1]$ and all entries of $x^i$ are in $[-r,r]$ and therefore $|W^1 x^i|\le nr < M_1$. Similarly we can show that $u_j^{i,1}=0$ if and only if $(w_j^1)^\top x^i < \lambda_1$. Hence $u^{i,1}$ is the output of the first layer for data point $x^i$ to which $W^2$ is applied in Constraints \ref{eq:MINLP_layers} and \ref{eq:MINLP_layers2}. By the same reasoning applied to Constraints \ref{eq:MINLP_layers} and \ref{eq:MINLP_layers2} we can show that $u^{i,k}$ is equal to the output of the $k$-th layer for data point $x^i$ for each $k\in [K]\setminus \{ 1\}$. Note that instead of the value $nr$ we can use $d_{k-1}$ here since the entries of $u^{i,k-1}$ can only have values $0$ or $1$ and the dimension of the rows of $W^k$ is $d_{k-1}$.
\end{proof}

The formulation in Lemma \ref{lem:MINLP_formulation} is a non-linear mixed-integer programming (MINLP) formulation, since it contains products of variables, where each is a product of a continuous variable and an integer variable. We can apply the classical McCormick linearization technique and replace each product of variables $w_{lj}^k u_j^{i,k-1}$ in the formulation of Lemma \ref{lem:MINLP_formulation} by a new variable $\mathop{s_{lj}^{i,k}\in [-1,1]}$. To ensure that 
\[
w_{lj}^k u_j^{i,k-1} = s_{lj}^{i,k}
\]
holds, we have to add the set of inequalities
\begin{align*}
& s_{lj}^{i,k} \le u_j^{i,k} \\
& s_{lj}^{i,k} \ge -u_j^{i,k}\\
& s_{lj}^{i,k} \le w_{lj}^k + (1-u_j^{i,k})\\
& s_{lj}^{i,k} \ge w_{lj}^k - (1-u_j^{i,k}) .
\end{align*}
Note that if $u_j^{i,k}=0$, then the first two constraints ensure that $s_{lj}^{i,k}=0$. Since $w_{lj}^k\in [-1,1]$ this combination is also feasible for the last two constraints. If $u_j^{i,k}=1$, then the last two constraints ensure, that $s_{lj}^{i,k}=w_{lj}^k$, while $s_{lj}^{i,k}$ and $u_j^{i,k}$ are still feasible for the first two constraints. Applying the latter linearization we can transform the formulation of Lemma \ref{lem:MINLP_formulation} into a mixed-integer linear program (MILP) if we use the empirical error described in \eqref{eq:empiricalLoss}. Note that the same linearization can be used if the neural network weights are integer variables, i.e. if we consider $D_k=\{-1,0,1\}^{d_k\times d_{k-1}}$. Fortunately, modern off-the-shelf MIP solvers as CPLEX or Gurobi can handle products of integer and continuous variables, hence the problem formulation of Lemma \ref{lem:MINLP_formulation} could directly be passed to the solver. For the regression variant with mean-squared error we obtain a quadratic mixed-integer program.

Unfortunately, the number of integer variables of the MIP formulation is of order $\mathcal O(DKm)$, where $D$ is the maximum dimension of the layers, and therefore grows linear with the number of training samples and with the number of layers. For practical applications requiring large training sets solving the MIP formulation can be a hard or even impossible task. To tackle these difficulties we propose two more efficient  heuristics in the following subsections. Nevertheless despite the computational challenges the MIP formulation has a lot of advantages and can give further insights into the analysis of deep neural networks:

\begin{itemize}
	
	
	\item More general discrete activation functions of the form
	\[
	\sigma^k(\alpha) = 
	v \ \text{ if } \underline{\lambda_k^{v}}\le \alpha \le \overline{\lambda_k^v} , \ v\in V\subset\mathbb Z
	\]
	for a finite set $V$ and pairwise non-intersecting intervals $[\underline{\lambda_k^{v}},\overline{\lambda_k^v} ]$ can be modeled by adding copies $u^{i,k,v}$ of the $u^{i,k}$ variables for each $v\in V$ and adding the constraints
	\begin{align*}
	& W^k \left( \sum_{v\in V} vu^{i,k-1,v}\right) \le M_k \left(1-u^{i,k,v}\right) + \lambda_k^{v} \\
	& W^k \left( \sum_{v\in V} vu^{i,k-1,v}\right) \ge -M_k \left(1-u^{i,k,v}\right) + \lambda_k^{v}
	\end{align*}
	for each $v\in V$, $k\in [K]\setminus\{1\}$ and $i\in[m]$. The two constraints for the first layer are defined similarly, replacing $\left( \sum_{v\in V} vu^{i,k-1,v}\right)$ by $x^i$. Note that the values $\underline{\lambda_k^{v}},\overline{\lambda_k^v}$ either have to be fixed in advance for each $v\in V$ or additional constraints have to be added which ensure the interval structure.

	\item The MIP formulation can easily be adjusted for applications where further constraints are desired. E.g. sparsity constraints of the form
	\[
	\| W^k\|_0\le q
	\]
	for an integer $q$ can be easily added to the formulation. Here $\| W^k\|_0$ is the number of non-zero entries of $W^k$. 
	
	\item Any classical approaches handling uncertainty in the data can be applied to the MIP formulation. In Section \ref{sec:adversarialAttacks} we will apply a robust optimization approach to the MIP formulation.
	
	\item The model is very flexible regarding changes in the training set. To add new data points that were not yet considered we just have to add the corresponding variables and constraints for the new data points to our already existing model and restart a solver, which is based on the idea of online machine learning. Furthermore, instead of adding all data points to the formulation, random batches could be used; see Section \ref{sec:data_splitting}. 
	
	\item Classical solvers like Gurobi use branch \& bound methods to solve MIP formulations. During these methods at each time the optimality gap, i.e. the percental difference between the best known upper and lower bound, is known. These methods can be stopped after a desired optimality gap is reached.
\end{itemize}

\subsection{Local Search Heuristic}\label{sec:heuristic}
Besides the Big-M constraints one of the main challenges of Problem \ref{eq:DNNDefinition} is the quadratic structure appearing in Constraints \ref{eq:MINLP_layers} and \ref{eq:MINLP_layers2}. While it is possible to use standard linearization techniques to derive a linear mixed-integer formulation, such transformations often do not result in efficiently solvable problem formulations. To this end, in this section we present an heuristic algorithm that is based on local search applied to the non-linear formulation in Lemma \ref{lem:MINLP_formulation}, also known under the name \textit{Mountain-Climbing method}; see \citet{nahapetyan2009bilinear}. The idea is to avoid the quadratic terms by alternately optimizing the MINLP formulation over a subset of the variables and afterwards over the complement of the variables. Since for given weight variables $W$ the $u$-variables define the activation patterns of the training data, exactly one feasible solution for the $u$-variables exists. Therefore the described local search procedure would terminate after one iteration if we iterate between weight variables and activation-variables. To avoid this case we go through all layers and alternately fix the $u$ or the $W$-variables. The complementary problem uses exactly the opposite variables for the fixation. More precisely, iteratively we solve first the problem formulation from Lemma \ref{lem:MINLP_formulation} where we replace Constraints \ref{eq:MINLP_variablesW} by 
\[
W^k\in D_k, \ \lambda_k\in [-1,1]
\qquad\qquad\forall k\in [K]: k \text{ odd}
\]
and replace Constraints \ref{eq:MINLP_variables} by 
\[
u^{i,k}\in \left\{ 0,1\right\}^{d_k} \qquad\qquad\forall i\in [m], \ k\in [K-1]: k \text{ odd}.
\]
We denote this problem by H1. Afterwards we solve the problem formulation from Lemma \ref{lem:MINLP_formulation} where we replace Constraints \ref{eq:MINLP_variablesW} by 
\[
 W^k\in D_k,  \ \lambda_k\in [-1,1]
\qquad\qquad\forall k\in [K]: k \text{ even}
\]
and replace Constraints \ref{eq:MINLP_variables} by 
\[
u^{i,k}\in \left\{ 0,1\right\}^{d_k} \qquad\qquad\forall i\in [m], \ k\in [K-1]: k \text{ even}.
\]
We denote this problem by H2.
In Problem H1 all variables in 
\begin{equation*}
	V_{\text{fix}}^1:=\{W^k,\lambda_k, u^{i,k} \text{ where $k\neq K$ and $k$ is even}\}
\end{equation*}
are fixed to the optimal solution values of the preceding Problem H2, while in Problem H2 all variables in 
\begin{equation*}
	V_{\text{fix}}^2:=\{W^k,\lambda_k, u^{i,k} \text{ where $k\neq K$ and $k$ is odd}\}
\end{equation*}
are fixed to the optimal solution values of the preceding Problem H1. Note that both problems are linear mixed-integer problems with roughly half of the variables of Problem (MIP). The heuristic is shown in Algorithm \ref{alg:heuristic}. Note that Algorithm \ref{alg:heuristic} returns a locally optimal solution and terminates after a finite number of steps, since there only exist finitely many possible variable assignments for the variables $u^{i,K}$, and therefore only finitely many objective values exist.

\begin{algorithm}\caption{(Local Search Heuristic)}\label{alg:heuristic}
\begin{algorithmic}
\Require $X\times Y$, $K$, $d_0,\ldots ,d_K$
\Ensure $W^1, \ldots, W^K$
	\State Draw random values for the variables in $V_{\text{fix}}^1$
	\While{no better solution is found} 
	\State Calculate an optimal solution of (H1), if feasible, for the current fixations in $V_{\text{fix}}^1$.
	\State Set all values in $V_{\text{fix}}^2$ to the corresponding optimal solution values of (H1).
	\State Calculate an optimal solution of (H2), if feasible, for the current fixations in $V_{\text{fix}}^2$.
	\State Set all values in $V_{\text{fix}}^1$ to the corresponding optimal values of (H2).
	\EndWhile
	\State Return: $W=\{ W^{k}\}_{k\in [K]}$
\end{algorithmic}
\end{algorithm}

We test Algorithm \ref{alg:heuristic} in Section \ref{sec:computations} on small datasets and show that in certain cases the derived neural networks outperform the exact method and classical DNNs of the same size on the Breast Cancer Wisconsin dataset.

\subsection{Iterative Data Splitting Algorithm}\label{sec:data_splitting}
In the last section we derived an efficient heuristic to avoid the quadratic structure in Problem \ref{eq:DNNDefinition}. Nevertheless another main challenge of Problem \ref{eq:DNNDefinition} is the large number of $0$-$1$ variables $u^{i,k}$ which grows linearly in the number of data points. Each variable $u^{i,k}_j$ models the activation of neuron $j$ in layer $k$ if data point $i$ is applied to the network. We say the neuron is \textit{activated} if $u^{i,k}_j=1$. An assignment of $0$-$1$ values to all neurons of the network is called an \textit{activation pattern}. Each activation pattern can be identified with a polyhedral region in the data space, given by Constraints \ref{eq:MINLP_firstlayer} -- \ref{eq:MINLP_variablesW} after fixing the $u$-variables to the given activation pattern. Similar observations were already made in \citet{rister2017piecewise,montufar2014number,wang2018max,raghu2017expressive,goerigk2020data}. Therefore, a valid assumption is that data points which are close to each other in the data space often follow the same activation pattern of a trained neural network. Indeed in \citet{goerigk2020data} it was observed that the number of different activation patterns of the training data is often very small. 

Using the latter observations, the idea of the following procedure is to iteratively split the training set $X$ into smaller subsets by using a distance-based clustering method and assign the same activation variables to all data points contained in one subset. More precisely for a partition $[m]=I_1\cup\ldots \cup I_p$ of the index set we define $X_{I_j}=\{x^i: i\in I_j\}$ for each $j=1,\ldots ,p$ and consider the problem 
\begin{equation}\label{eq:MINLP_reducedUVariables}
\begin{aligned}
& \min \ \sum_{j=1}^{p}\sum_{i\in I_j} \ell\left( y^i , u^{I_j,K}\right) \quad s.t. \ \  \\
& W^1 x < M_1 u^{I_j,1} + \1\lambda_1 \quad \forall \ x\in X_{I_j} \ ,j\in [p]\\
& W^1 x \ge M_1 (u^{I_j,1}-1) + \1\lambda_1 \quad \forall \ x\in X_{I_j} \ ,j\in [p]\\
& W^k u^{I_j,k-1} < M_k u^{I_j,k} + \1\lambda_k  \ \ \forall k\in [K]\setminus \{1\}, j\in [p]\\
& W^k u^{I_j,k-1} \ge M_k(u^{I_j,k}-1) + \1\lambda_k  \ \ \forall k\in [K]\setminus \{1\}, j\in [p]\\
& W^k \in[-1,1]^{d_k\times d_{k-1}}, \ \ \lambda_k\in[-1,1] \ \ \forall k\in [K] \\ 
& u^{I_j,k}\in \left\{ 0,1\right\}^{d_k} \ \ \forall k\in [K], j\in [p]
\end{aligned}
\end{equation}
instead of Problem \ref{eq:DNNDefinition}. The idea here is that each data point $x\in X_{I_j}$ has to follow the same activation pattern which is determined by variables $u^{I_j,k}$. Note that the number of $u$ variables reduces from $\mathcal O(Km)$ to $\mathcal O(Kp)$ and hence can be controlled by the parameter $p$. Partitions of the index set $[m]$ are derived by iteratively splitting the subset $X_{I_j}$ which contains the largest number of misclassified data points by using $k$-means clustering. After each split we train the model by solving \eqref{eq:MINLP_reducedUVariables} for the new partition. For each data point in $X$ we have two constraints related to the first layer, hence to reduce the number of constraints, we consider a random batch of data points in each iteration. The procedure is shown in Algorithm \ref{alg:iterativemethod}.  

\begin{algorithm}\caption{(Iterative Data-Splitting Algorithm)}\label{alg:iterativemethod}
\begin{algorithmic}
\Require $X\times Y$, $K$, $d_0,\ldots ,d_K$, epochs $T$, batchsize $b$
\Ensure $W^1, \ldots, W^K$
	\State set $t=1$, $\mathcal I = \{ [m]\}$
	\While{$t\le T$} 
	\State Solve \eqref{eq:MINLP_reducedUVariables} with partition $\mathcal I$ and for a random batch of size $b$.
	\State Calculate $I_{\text{max}}\in \argmax_{I\in \mathcal I} \sum_{i\in I} \ell (y^i,u^{I,K})$
	\State Split $X_{I_{\text{max}}}$ into two clusters $X_{I_1}$ and $X_{I_2}$ using $k$-means.
	\State Remove $I_{\text{max}}$ from $\mathcal I$ and add $I_1$ and $I_2$
	\EndWhile
	\State Return: $W=\{ W^{k}\}_{k\in [K]}$
\end{algorithmic}
\end{algorithm}

Note that each iteration in Algorithm \ref{alg:iterativemethod} can be interpreted as the counterpart to an epoch of stochastic gradient descent. In each iteration Problem \ref{eq:MINLP_reducedUVariables} has a significantly smaller number of variables than the exact problem and can be solved e.g. by using a standard IP solver as CPLEX or Gurobi. The $k$-means procedure ensures that each derived subset in $\mathcal I$ contains data points which are close to each other. The larger the number of iterations, the finer the partition and therefore the better the loss of the BDNN. However since we draw random batches of data points in each iteration the loss of the subsequent iteration does not necessarily get smaller. Note that Problem \ref{eq:MINLP_reducedUVariables} is feasible for every given partition of the index set, since we can always set all network weights to $0$ and therefore each data point has the same activation pattern.


We test Algorithm \ref{alg:iterativemethod} in Section \ref{sec:computations} on several datasets and show that even for large datasets Algorithm \ref{alg:iterativemethod} often returns networks with high accuracy in reasonable time.

\section{Robust Binarized Neural Networks under Data Uncertainty}\label{sec:adversarialAttacks}
In this section we consider labeled data \[X\times Y=\left\{ (x^i,y^i) \ | \ i\in [m]\right\}\] where the Euclidean norm of each data point is bounded by $r>0$ and the data points are subject to uncertainty, i.e. a true data point $x^i\in\R^n$ can be perturbed by an unknown deviation vector $\delta\in\R^n$. In this case one goal is to derive BDNNs which are robust against data perturbation, i.e. if we add a small perturbation to the data point the predicted class of the BDNN should not change. One approach to tackle data uncertainty is based on the idea of robust optimization and was already studied for linear regression problems and support vector machines in \citet{bertsimas2019robust,xu2009robustness,xu2009robust}. IP models for the verification of the robustness of already trained deep neural networks were studied in \citet{tjeng2017evaluating,khalil2018combinatorial,venzkeRobustDNN}. Furthermore there are several approaches which use adversarial attacks to robustify the networks during training, without achieving a robustness guarantee \citep{venzke2020verification,kurakin2016adversarial,yuan2019adversarial}. In the following we derive the first model which ensures robustness during training of binarized neural networks. 

In the robust optimization setting, we assume that for each data point $x^i$ we have a convex set of possible deviation vectors $U^i\subset\R^n$ called \textit{uncertainty set} which is defined by
\[
U^i = \left\{ \delta\in \R^n \ | \ \|\delta\|\le r_i\right\}
\]
for a given norm $\|\cdot \|$ and radii $r_1,\ldots ,r_m$. Note that classical convex sets like boxes, polyhedra, or ellipsoids can be modeled as above by using the $\ell_\infty, \ell_1$ or $\ell_2$ norm. The task is to calculate the weights of a neural network which is robust against all possible data perturbations contained in the uncertainty set $U:=U^1\times \cdots\times U^m$, i.e. the predicted class of the neural network has to be the same for all perturbations in the uncertainty set. While the derivation of efficient robust counterparts for SVMs or linear regression problems is possible due to the simple structure of the problems, the non-linearity of neural networks makes this task much more difficult. As mentioned above robustness of an already trained neural network can be either tested via integer programming models or can be enforced during training by adding attacked data points to the training set which does not yield a robustness guarantee for the whole uncertainty set $U$. However due to the combinatorial structure of the binary activation functions we are able to ensure robustness already during training of the BDNN. This can be achieved by the following idea: if the activation values of the first layer neurons is the same for all possible perturbations of a data point, then the input of the second layer is the same $0$-$1$ vector for each possible perturbation of the data point and  therefore the output of the BDNN is the same. As a consequence, if we can ensure robustness in the first layer we obtain robustness of the whole network.

To ensure robustness in the first layer we have to ensure that for each data point $x^i$ and each possible perturbation in $U^i$ the activation variable $u^{i,1}$ does not change. This can be guaranteed by the following robust constraints:
\begin{align*}
& \max_{\delta\in U^i} ~W^1 (x^i+\delta) <  M_1^i u^{i,1} + \lambda_1 \\
& \min_{\delta\in U^i} ~W^1 (x^i+\delta) \ge M_1^i (u^{i,1}-1) + \lambda_1 .
\end{align*}
Note that for each feasible activation pattern $u^{i,1}\in\{0,1\}^{d_1}$ the latter constraints ensure that $x^i+\delta$ has the same activation pattern for every perturbation $\delta\in U^i$. A classical result from linear robust optimization is that we can consider the maximum and the minimum expression constraint-wise. We can reformulate the left-hand sides as
\[
\max_{\delta\in U^i} ~(w_j^1)^\top (x^i+\delta) = (w_j^1)^\top x^i + \max_{\delta\in U^i} ~(w_j^1)^\top\delta = (w_j^1)^\top x^i + r_i\|w_j^1\|^*
\] 
where $w_j^1$ is the $j$-th row of matrix $W^1$ and $\|\cdot\|^*$ is the dual norm of $\| \cdot \|$. Equivalently we can reformulate the left-hand sides of the second constraints as
\[
\min_{\delta\in U^i} ~(w_j^1)^\top (x^i+\delta) = (w_j^1)^\top x^i - r_i\|w_j^1\|^*
\] 
and therefore the robust BDNN model is given by
\begin{equation}\label{eq:robustBDNN}
\begin{aligned}
& \min \ ~\sum_{i=1}^{m} \ell\left( y^i , u^{i,K}\right) \quad s.t. \\
& (w_j^1)^\top x^i + r_i\|w_j^1\|^* <  M_1^i u^{i,1} + \lambda_1 \ \ \forall i\in [m], \ j\in [d_1] \\
& (w_j^1)^\top x^i - r_i\|w_j^1\|^*\ge M_1^i (u^{i,1}-1) + \lambda_1 \ \ \forall i\in [m], \ j\in [d_1]\\
&W^k u^{i,k-1} < M_k u^{i,k} + \1\lambda_k \ \ \forall i\in [m], k\in [K]\setminus \{1\}\\
&W^k u^{i,k-1} \ge M_k(u^{i,k}-1) + \1\lambda_k \ \ \forall i\in [m], k\in [K]\setminus \{1\}\\
& W^k\in D_k, \lambda_k\in [-1,1] \ \ \forall k\in [K]\\
&u^{i,k}\in \left\{ 0,1\right\}^{d_k} \ \ \forall i\in [m], \ k\in [K] 
\end{aligned}
\end{equation}
where $M_1^i:= n(r+r_i)$ and $M_k$ is defined as in Section \ref{sec:DNN}. Note that if we choose the euclidean norm, then we obtain a quadratic problem, while if we choose the $\ell_\infty$ or $\ell_1$-norm the latter problem can be transformed into a linear problem. In practical applications one of the main questions is how to choose the size of $U$, i.e. the radii $r_i$. In Section \ref{sec:computations} we test the robust model \eqref{eq:robustBDNN} for several magnitudes of uncertainty sets and attacks.

\section{Computations}\label{sec:computations}
In this section we first perform experiments on small datasets to test the exact model, denoted by BDNN, and the local search heuristic, denoted by LS (see Section \ref{sec:DNN} and \ref{sec:heuristic}). Since the computability of both models does not scale well with increasing input parameters we can only perform experiments on small datasets, small network architectures and without considering integer weights. We study two variants, one where we fix the thresholds $\lambda_k$ and another where we determine the threshold values during training. On the other hand in the second subsection we test the iterative data-splitting algorithm (see Section \ref{sec:data_splitting}), denoted by DS, which performs much better and could be tested on larger datasets, larger network architectures and with integer and continuous weights. Finally in the third subsection we test the robust BDNN version presented in Section \ref{sec:adversarialAttacks}, solved by the data-splitting algorithm, denoted by RO-BDNN. An overview about the considered datasets can be found in Table \ref{tbl:datasets}. Our Python code related to the experiments is made available online\footnote{\href{https://github.com/JannisKu/BDNN2021}{https://github.com/JannisKu/BDNN2021}}.

\begin{table}[h!]
\caption{Description of the considered datasets.}
\label{tbl:datasets}
\begin{center}
\begin{tabular}{l|llll}
Dataset & \# inst. & \# attr. & \# classes & class distr.  \\
\hline
Breast Cancer Wisconsin (BCW) & $699$ & $9$ & $2$ & $65.5\%$/$34.5\%$ \\
Default Credit Card (DCC) & $30000$ & $23$ & $2$ & $77.9\%$/$22.1\%$ \\
Iris & $150$ & $4$ & $3$ & $33.3\%$/$33.3\%$/$33.3\%$ \\
Boston Housing (BH) & $506$ & $13$ & $2$ & $50.6\%$/$49.4\%$ \\
Digit Dataset (DD) & $1797$ & $64$ & $10$ & $10\%$/\ldots /$10\%$
\end{tabular}
\end{center}
\vspace{-3mm}
\end{table}

\subsection{Exact model and local search heuristic}
In this section, we investigate the exact model presented in Lemma \ref{lem:MINLP_formulation} (BDNN) and the local search heuristic (LS) presented in Section \ref{sec:heuristic} . Both models are studied for continuous weights, i.e. $W^k\in [-1,1]^{d_k}$ and for two variants, one where the thresholds $\lambda_k$ are fixed to $0$ (denoted as BDNN$_0$) and another where the thresholds are derived during training (denoted as BDNN). We computationally compare both methods for the BDNN to the classical DNN with the ReLU activation functions. We study networks with one hidden layer of dimension $d_1$. All solution methods were implemented in Python 3.8 on an Intel(R) Core(TM) i5-4460 CPU with 3.20GHz and 8 GB RAM. The classical DNN was implemented by using the Keras API where we used the ReLU activation function on the hidden layer and the Softmax on the output layer. We used the binary cross entropy loss function. The number of epochs was set to $100$. The exact IP formulation is given in Lemma \ref{lem:MINLP_formulation} and all IP formulations used in the local search heuristic were implemented in Gurobi 9.0 with standard parameter settings. The strict inequalities in the IP formulations were replaced by non-strict inequalities adding $-0.0001$ to the right-hand-side. For the IP formulations, we set a time limit (wall time) of 24 hours.

\begin{figure}[h!]
\centering
\includegraphics[scale=0.4]{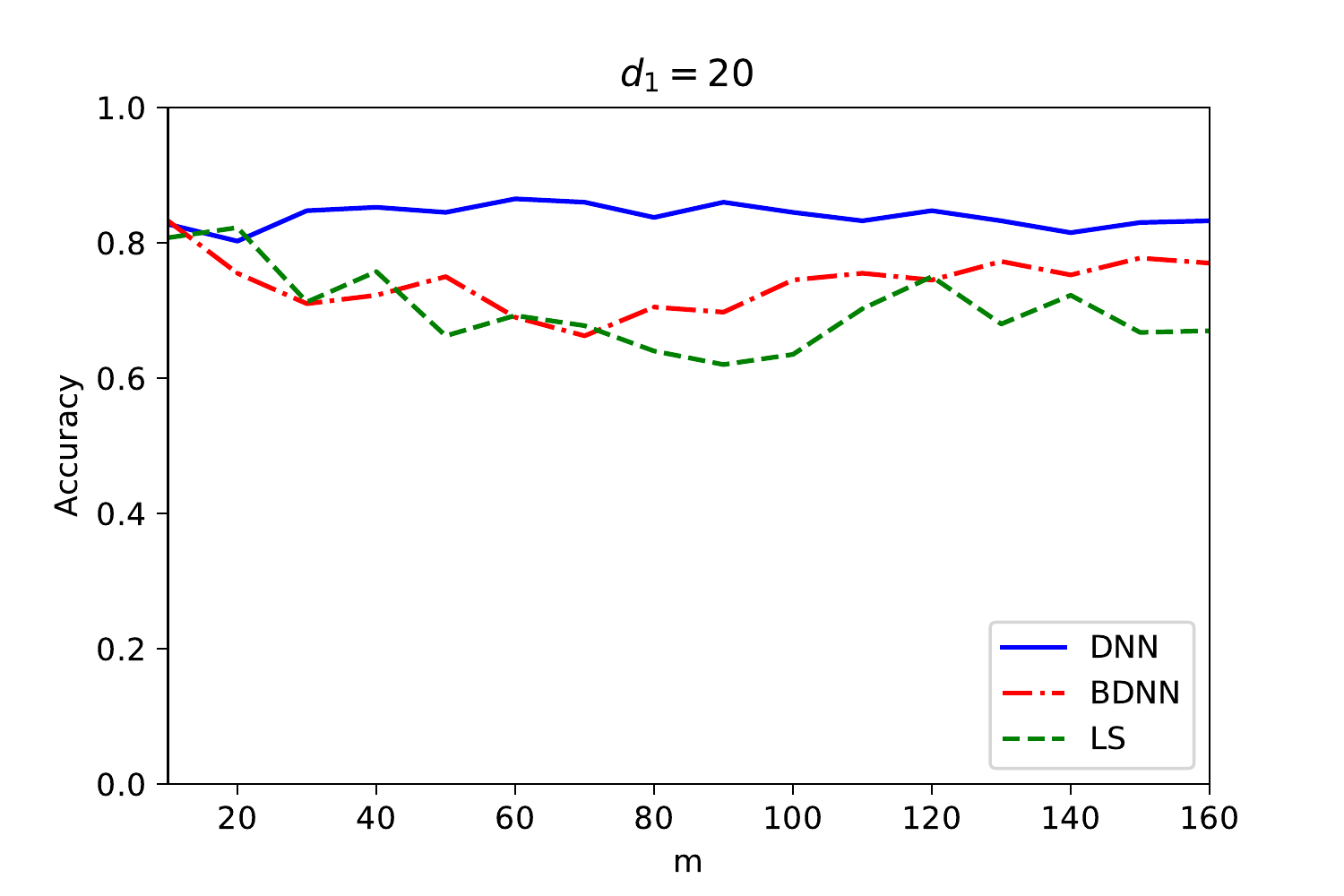}
\includegraphics[scale=0.4]{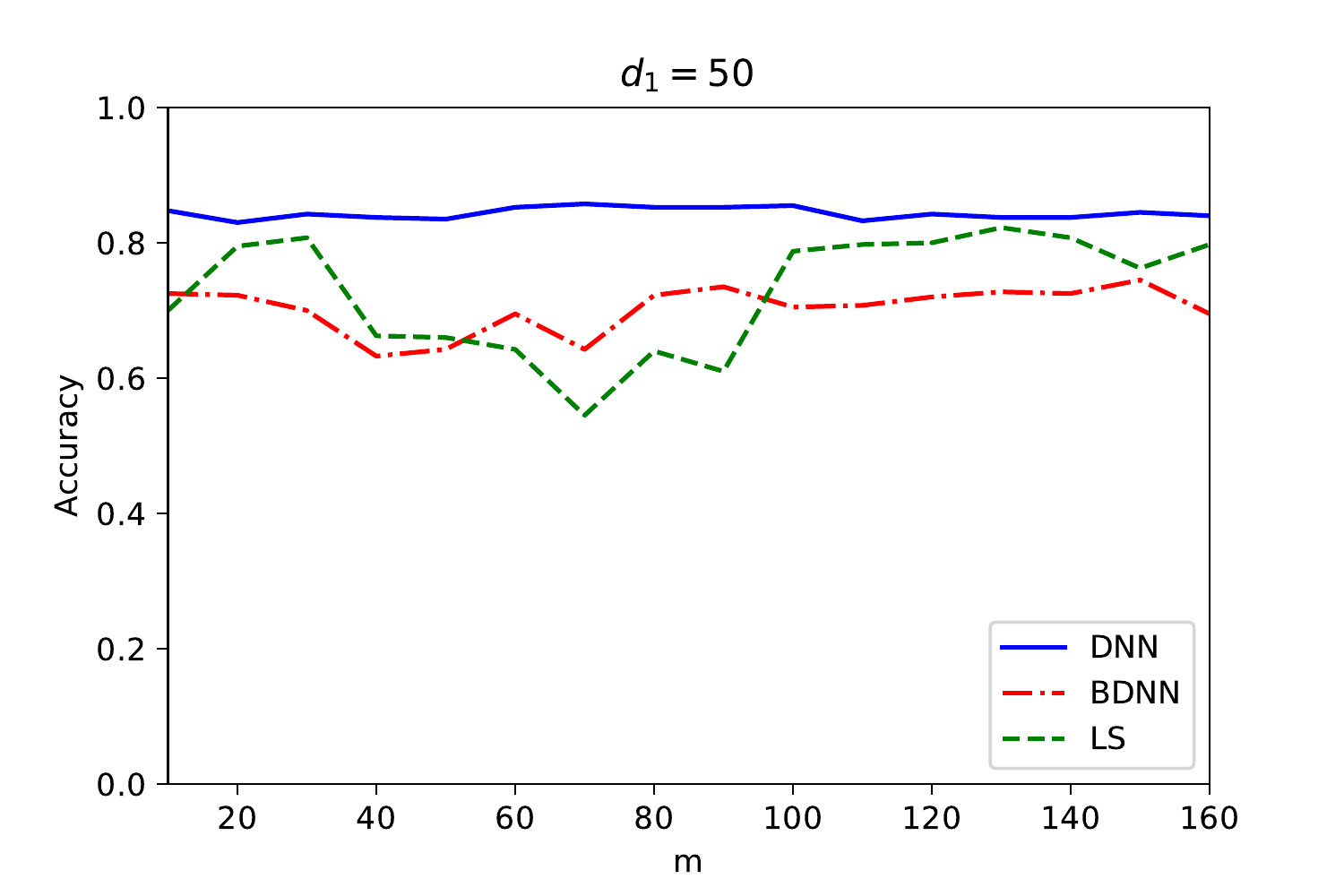}
\includegraphics[scale=0.4]{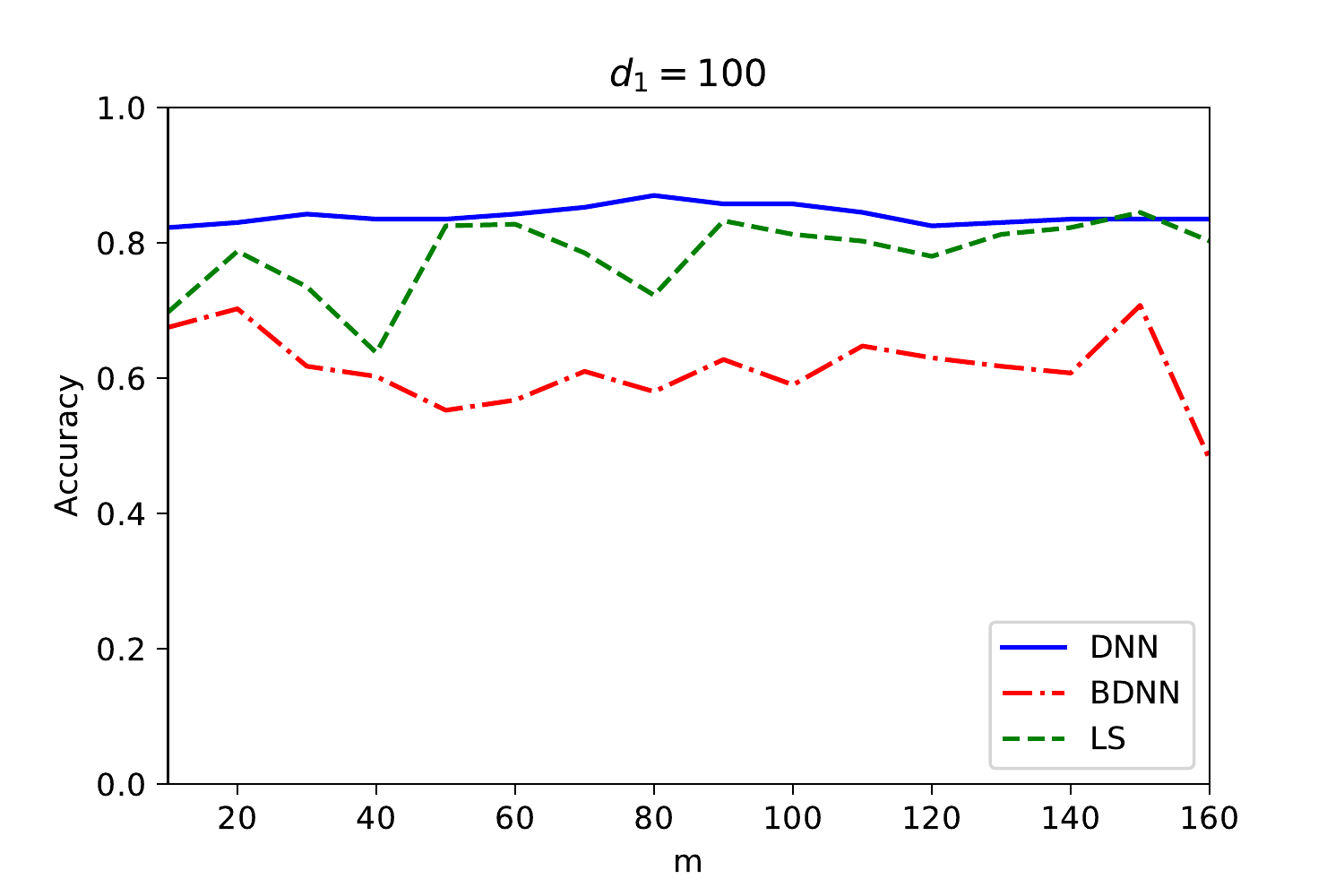}
\caption{Average accuracy over $10$ random instances for networks with one hidden layer of dimension $d_1\in \{ 20,50,100\}$ trained on $m\in \{10, 20, \ldots,160\}$ data points.}
\label{fig:plotsAcc}
\vspace{-3mm}
\end{figure}

We generated $10$ random datasets in dimension $n=100$ each with $M=200$ data points and $m\in \{10, 20, \ldots,160\}$ training samples. The entries of the data points were drawn from a uniform distribution with values in $[0,10]$ for one-third of the data points, having label $1$, and with values in $[-10,0]$ for the second third of the data points, having label $0$. The remaining data points were randomly drawn with entries in $[-1,1]$ and have randomly assigned labels. We split each dataset into a training set of $m\in \{10, 20, \ldots,160\}$ samples and a testing set of $40$ samples. All computations were implemented for neural networks with one hidden layer of dimension $d_1\in \{ 20,50,100\}$. Figure \ref{fig:plotsAcc} shows the average classification accuracy on the testing set over all $10$ datasets achieved by the methods trained on $m$ of the training points. The results indicate that the exact BDNN and LS have a lower accuracy than the DNN. Furthermore, the performance of both BDNN methods seem to be much more unstable and depend more on the choice of the training set. Interestingly for a hidden layer of dimension $100$, LS performs better than the exact version and can even compete with the classical DNN. In Figure \ref{fig:plotsTime} we show the runtime of all methods over $m\in \{10, 20, \ldots,160\}$. Clearly, the runtimes of the BDNN methods are much higher and seem to increase linearly in the number of data points. For real-world datasets with millions of data points, both methods will fail using state-of-the-art solvers. Surprisingly, the runtime of LS seems to be nearly the same as for the exact version, while the accuracy can be significantly better. We argue that local optima of the LS seem to perform better in terms of accuracy.

\begin{figure}[h!]
\centering
\includegraphics[scale=0.4]{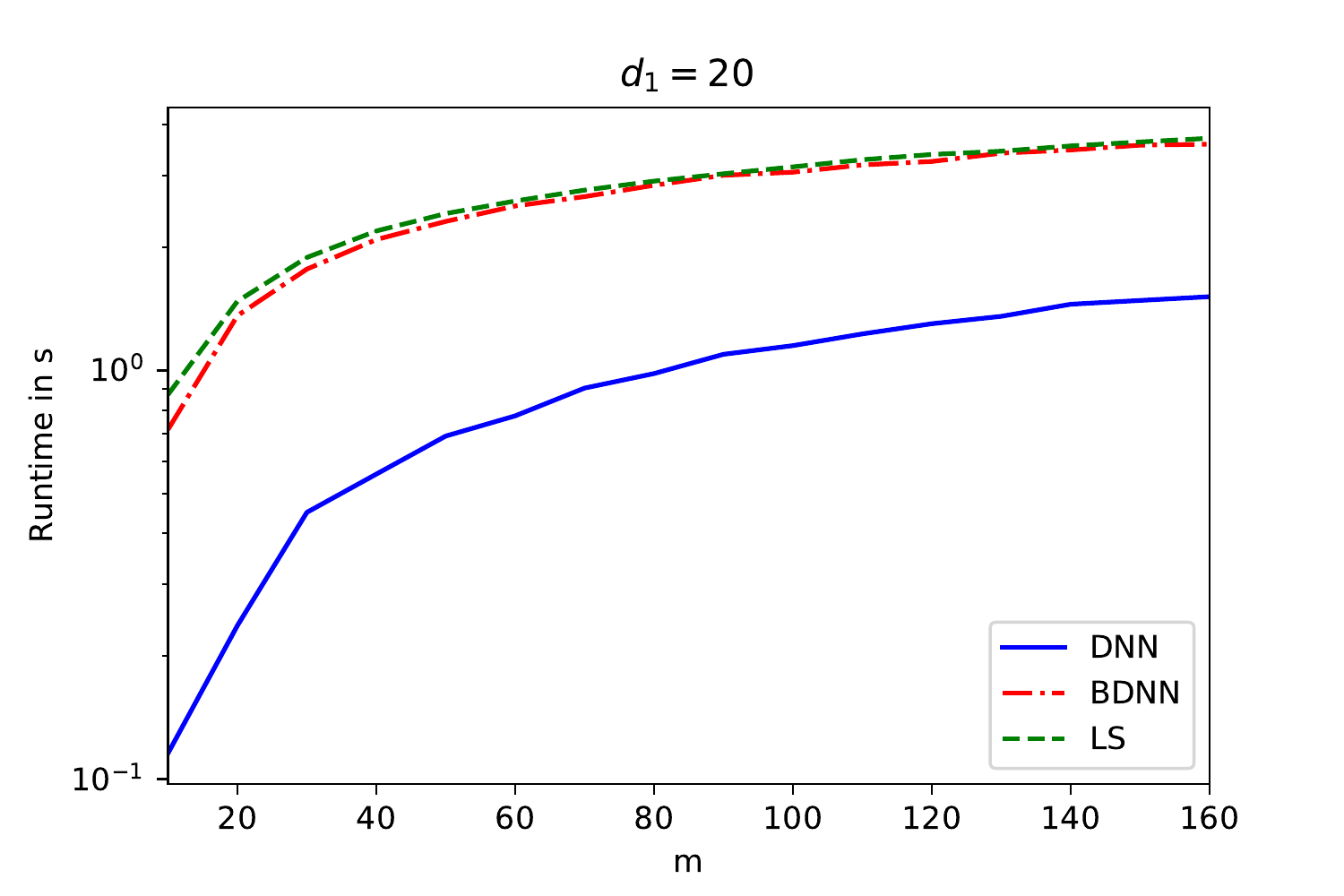}
\includegraphics[scale=0.4]{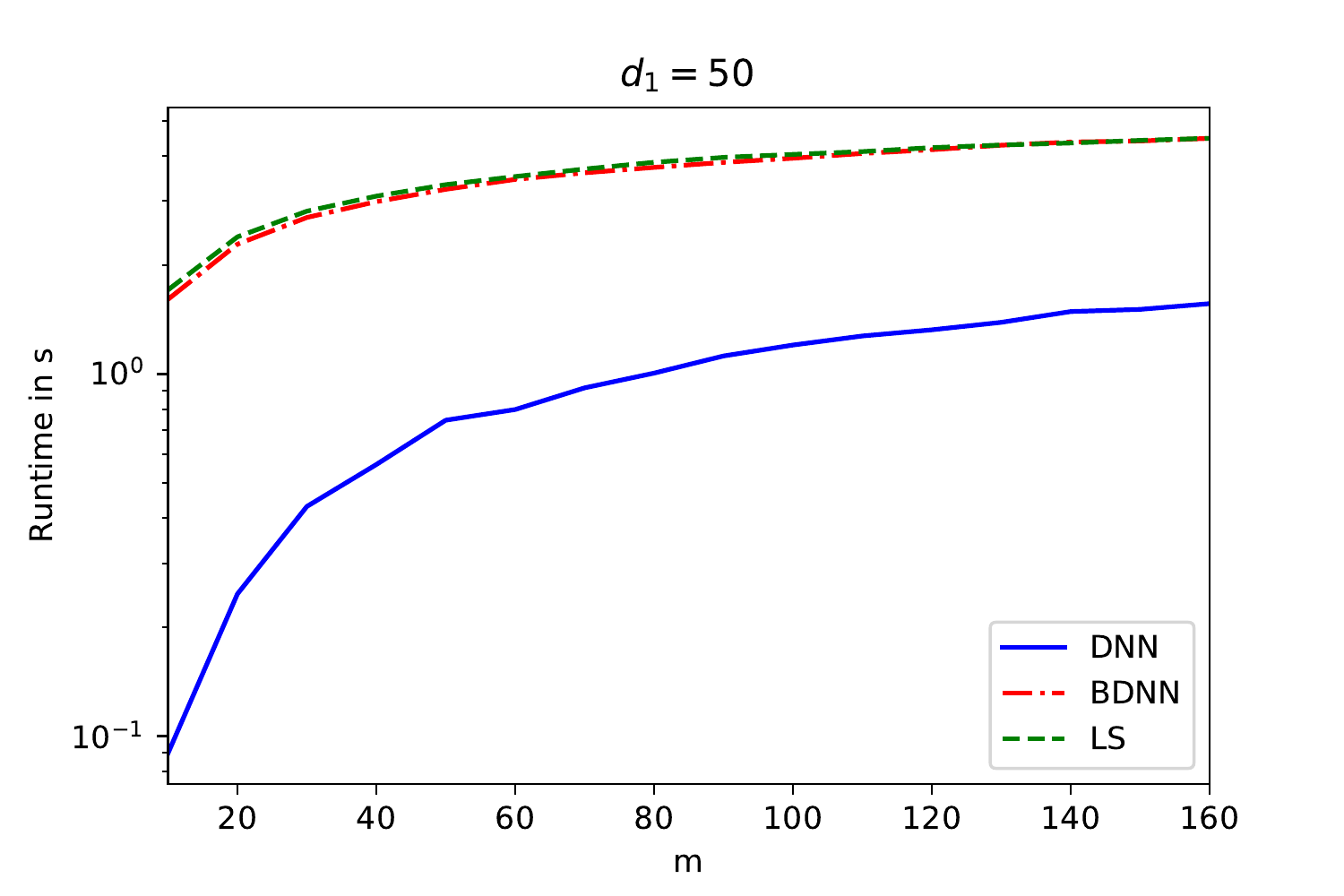}
\includegraphics[scale=0.4]{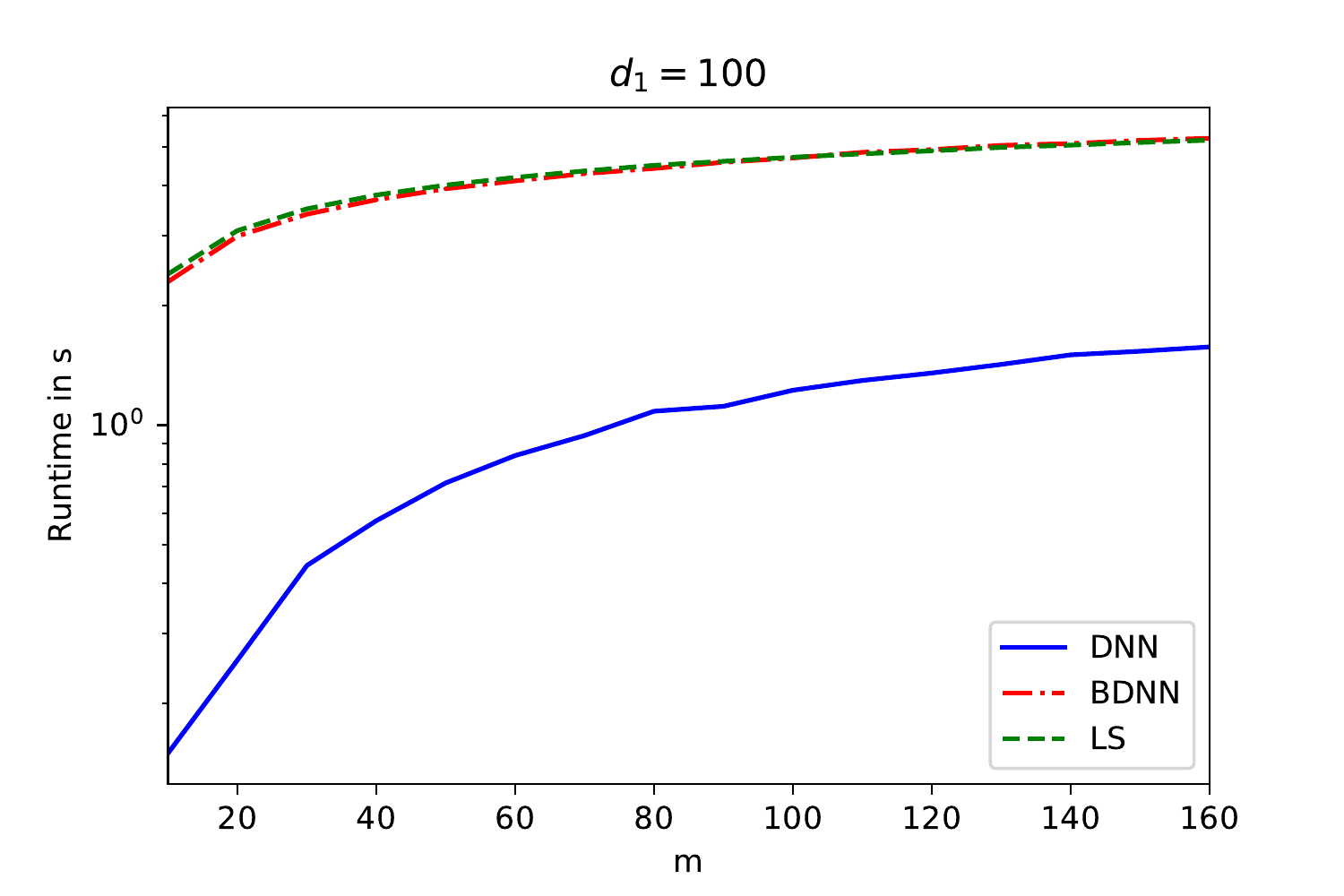}
\caption{Average runtime (logarithmic scale) over $10$ random instances for networks with one hidden layer of dimension $d_1$ trained on $m$ data points.} 
\label{fig:plotsTime}
\vspace{-3mm}
\end{figure}

Additionally, we study all methods on the Breast Cancer Wisconsin dataset (BCW) \citep{Dua:2019}. Here we also test the BDNN version where the values $\lambda_k$ are all set to $0$ instead of being part of the variables; we indicate this version by BDNN$_0$. The dataset was split into $80\%$ training data and $20\%$ testing data.  Again all computations were implemented for neural networks with one hidden layer of dimension $d_1\in \{ 25,50\}$. In Table \ref{tbl:breastcancerAcc} we show the accuracy, precision, recall and F1 score of all methods on the BCW dataset for a fixed shuffle of the data returned by the scikit-learn method \textit{train\_test\_split} with the seed set to $42$. It turns out that the exact BDNN performs better if the values $\lambda_k$ are set to $0$, while LS performs better if the $\lambda_k$ are trained. The LS method has the best performance for $d_1=25$, significantly better than the DNN. For $d_1=50$ the DNN is slightly better. Nevertheless, the best accuracy of $95\%$ for the BCW dataset was achieved by the LS method for $d_1=25$. Additionally we report the optimality gap after the time limit in Table \ref{tbl:breastcancerAcc} given by Gurobi. Only the BDNN$_0$ could not be solved to optimality during the time limit having a small gap of $0.51\%$. In Table \ref{tbl:breastcancerShuffles} we compare LS to the DNN on $10$ random shuffles of the BCW dataset and record the average, maximum, and minimum accuracies over all $10$ shuffles. It turns out that LS outperforms the DNN with the best accuracy of $97.1\%$ which leads to the conclusion that for small networks the BDNN can compete with the DNN in terms of the accuracy metric. Nevertheless as a drawback the training times for the BDNN methods are very large.

\begin{table}[h!]
\caption{Performance of exact BDNN and LS on the \textit{Breast Cancer Wisconsin} dataset.}
\label{tbl:breastcancerAcc}
\vskip 0.15in
\begin{center}
\begin{tabular}{lc|ccccc}
Method & $d_1$ & Acc. (\%) & Prec. (\%) & Rec. (\%) & F1 (\%) & Opt. Gap (\%)  \\
\hline
BDNN    & 25 & 69.3 & 48.0 & 69.3 & 56.7 & 0.0 \\
BDNN$_0$ & 25 & 83.6 & 83.2 & 83.6 & 83.1 & 0.0\\
BDNN LS & 25 & \textbf{95.0} & \textbf{95.0} & \textbf{95.0} & \textbf{95.0} & 0.0 \\
BDNN$_0$ LS & 25 & 30.0 & 38.7 & 30.0 & 17.6 & 0.0 \\
DNN    & 25 & 91.4 & 91.8 & 91.4 & 91.5 & - \\
\hline
BDNN    & 50 & 69.3 & 63.6 & 69.3 & 58.0 & 0.0 \\
BDNN$_0$ & 50 & 84.3 & 86.4 & 84.3 & 84.7 & 0.51\\
BDNN LS & 50 & 89.3 & \textbf{91.5} & 89.3 & 89.6 & 0.0 \\
BDNN$_0$ LS & 50 & 71.4 & 74.6 & 71.4 & 72.3 & 0.0 \\
DNN    & 50 & \textbf{91.4} & 91.4 & \textbf{91.4} & \textbf{91.3} & - \\
\end{tabular}
\end{center}
\vspace{-3mm}
\end{table}

%

\begin{table}[h!]
\caption{Accuracy (in \%) of the local search heuristic on the \textit{Breast Cancer Wisconsin} dataset over $10$ random shuffles of the data.}
\label{tbl:breastcancerShuffles}
\vskip 0.15in
\begin{center}
\begin{tabular}{lc|ccc}
Method & $d_1$ &  Avg. (\%) & Max (\%) & Min (\%)  \\
\hline
BDNN LS & 25 & \textbf{93.2} & \textbf{97.1} & 85.0\\
DNN    & 25 & 89.1 & 91.4 & \textbf{85.7}\\

\end{tabular}
\end{center}
\vspace{-3mm}
\end{table}

\subsection{Iterative Data-Splitting Algorithm}
In this Section we study the iterative data-splitting algorithm (DS) presented in Section \ref{sec:data_splitting}. As the experiments in the latter subsection show, the exact model and the local search heuristic are very hard to solve, and it is too costly to apply these methods on larger datasets, for larger networks or with integer weights. In this section we show that DS can be applied to larger datasets, neural networks with up to $3$ layers and can be even solved in reasonable time if we consider integer weights. At the same time this model often leads only to small reductions in accuracy or even outperforms the classical DNNs on small network architectures.

In the following we consider neural networks with $K\in\{1,2,3\}$ layers, a consistent layer width of $d_k\in\{ 50,100\}$ and bias vectors in each of the layers. For the DS algorithm we run $T=20$ epochs and use a batch size of $b=32$ for Iris and BCW datasets, while we use a batch size of $b=64$ for DCC, BH and DD datasets. For the data-splitting we use the $k$-means implementation of the scikit-learn package and Problem \ref{eq:MINLP_reducedUVariables} is solved by Gurobi 9.0. Each dataset is split into $50\%$ training samples, $25\%$ validation samples and $25\%$ testing samples. After each iteration of the DS algorithm we test the derived BDNN on the validation set and save the network which has the best validation accuracy over all epochs. Afterwards we calculate the accuracy of the best network on the testing set. 

Again we use the Keras API to train the DNNs, where we use ReLU activation functions on each hidden layer and the Softmax function on the output layer. We allow bias terms in each layer and use the categorical cross entropy loss function.  We consider the same number of epochs and the same batch sizes as for the DS calculations. As in the DS algorithm after each epoch we test the current DNN on the validation set and save the best network which is then evaluated on the test samples.  

All solution methods were implemented in Python 3.8 on a node with two AMD EPYC 7452 CPUs with 32 cores, 2.35-3.35 GHz, 128 MB Cache and a RAM of 256 GB DDR4 and 3200MHz.
\begin{table}[h!]
\caption{Accuracy (in \%) and runtime (in seconds) for BCW and DCC datasets.}
\label{tbl:splitAlgo_breast_cancer}
\begin{center}
\begin{tabular}{c|cc|ll|ll|ll}
&&&\multicolumn{2}{c|}{DNN} & \multicolumn{4}{c}{BDNN}\\
&&&& & \multicolumn{2}{c|}{integer weights} & \multicolumn{2}{c}{cont. weights}\\
Dataset & \# hidden layers & width & Acc. & $t$  & Acc.  & $t$  & Acc.  & $t$  \\
\hline
BCW&1&50&90.7&1.4&\textbf{96.9}&272.8&95.7&160.5 \\
&1&100&93.7&1.3&95.5&394.9&\textbf{97.1}&287.4\\
&2&50&92.6&1.4&\textbf{96.3}&1224.1&83.4&462.9\\
&2&100&94.0&1.5&\textbf{95.2}&1459.4&84.8&762.8\\
&3&50&94.5&1.6&\textbf{95.5}&3497.7&73.2&602.4\\
&3&100&\textbf{95.0}&1.7&93.0&3513.7&75.8&1031.2 \\
\hline
DCC&1&50&\textbf{78.1}&5.5&\textbf{78.1}&171.5&77.9&172.4 \\
&1&100&\textbf{78.0}&5.5&\textbf{78.0}&331.5&77.9&329.9\\
&2&50&77.9&5.6&\textbf{78.1}&186.6&\textbf{78.1}&187.7\\
&2&100&77.6&6.1&\textbf{77.7}&382.4&\textbf{77.7}&384.3\\
&3&50&77.5&6.7&77.8&203.0&\textbf{77.9}&203.4\\
&3&100&77.7&10.8&\textbf{77.9}&437.7&\textbf{77.9}&436.2
\end{tabular}
\end{center}
\vspace{-3mm}
\end{table}

In Table \ref{tbl:splitAlgo_breast_cancer} we consider BCW and DCC datasets and study classical DNNs and two variants of the BDNN, one with continuous weights $W^k\in [-1,1]^{d_k}$ and another with integer weights $W_k\in \{-1,0,1\}^{d_k}$. For all methods we report the average accuracy on the testing set and the runtime in seconds for each network architecture over $10$ random train-validation-test splits of the dataset. The results indicate that the BDNN with integer weights performs much better than the same model with continuous weights. On the BCW dataset the BDNN with integer weights even outperforms the classical DNNs in terms of accuracy on most of the network architectures. Only on the largest architecture the DNN achieves the best accuracy. Note that the best accuracy was achieved by the BDNN with continuous weights and a network with $1$ hidden layer of width $100$. On the other hand the computation time of both BDNN methods is much larger (up to one hour) than the computations for the DNN (at most 11 seconds). Here the computation times of the BDNN with continuous weights can be significantly smaller than for the same model with integer weights. On the DCC dataset all methods fail to learn any data information since predicting always the first class would achieve an accuracy of around $78\%$. 

\begin{table}[h!]
\caption{Accuracy (in \%) and runtime in seconds for various datasets.}
\label{tbl:splitAlgo_otherDatasets}
\begin{center}
\begin{tabular}{c|cc|ll|ll}
&&&\multicolumn{2}{c|}{DNN} & \multicolumn{2}{c}{BDNN}\\
Dataset & \# hidden layers & width & Acc. & $t$  & Acc.  & $t$ \\
\hline
Iris&1&50&40.3&1.1&\textbf{87.8}&2748.3 \\
&1&100&58.2&1.1&\textbf{92.1}&2483.7\\
&2&50&54.6&1.0&\textbf{68.1}&2101.2\\
&2&100&\textbf{87.0}&1.0&53.1&3862.7\\
&3&50&\textbf{87.4}&1.3&45.3&6460.7\\
&3&100&\textbf{85.0}&1.4&73.8&4779.1 \\
\hline
BH&1&50&73.1&1.4&\textbf{73.9}&774.7 \\
&1&100&\textbf{73.9}&1.3&73.1&1174.3\\
&2&50&\textbf{75.6}&1.1&60.4&858.6\\
&2&100&\textbf{76.1}&1.1&61.7&1606.9\\
&3&50&\textbf{78.0}&1.3&55.1&1150.7\\
&3&100&\textbf{77.0}&1.2&57.2&2240.1 \\
\hline
DD&1&50&\textbf{95.7}&1.9&72.0&3268.6 \\
&1&100&\textbf{97.2}&2.1&76.6&4433.5\\
&2&50&\textbf{95.8}&2.2&60.6&20259.3\\
&2&100&\textbf{98.5}&2.2&63.8&28238.2\\
&3&50&\textbf{96.2}&2.2&31.8&31548.6\\
&3&100&\textbf{98.8}&2.2&18.3&33321.3
\end{tabular}
\end{center}
\vspace{-3mm}
\end{table}

In Table \ref{tbl:splitAlgo_otherDatasets} we consider the Iris, BH and DD datasets and show the average performance of the DNN and the DS algorithm for BDNNs with integer weights, since the results in Table \ref{tbl:splitAlgo_breast_cancer} indicate that the latter performs better than the one with continuous weights. The results indicate that the DS algorithm outperforms the DNN in terms of accuracy on small network architectures for the BH and Iris datasets while it performs significantly worse on larger architectures. Nevertheless on the Iris dataset the overall best accuracy is achieved by the BDNN on a network with $1$ hidden layer of size $100$. On the other hand again the computation times of the BDNN are significantly larger. For the DD dataset the BDNN has a much smaller accuracy than the DNN. Furthermore the accuracy of the BDNN decreases significantly with increasing network size, while the accuracy of the DNN remains stable. Here we can assume that the BDNN needs a much larger number of data-splits due to the larger number of classes (10 classes). Furthermore due to the large number of features the computations for this dataset are very expensive. In summary the results show that we can train BDNNs with our integer programming model on much larger datasets and larger network architectures than state-of-the-art (note that no computations for integer programming models on the same datasets and the same network architecture were performed yet) and the accuracy values indicate that we can achieve high accuracy on small network architectures with the BDNN. Hence BDNNs are a reasonable choice when the memory consumption is restricted.

\begin{figure}[h!]
\centering
\includegraphics[scale=0.4]{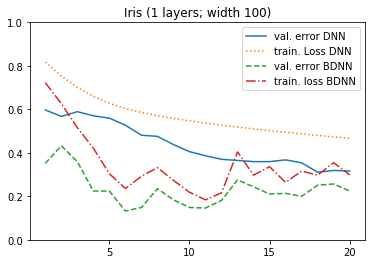}
\includegraphics[scale=0.4]{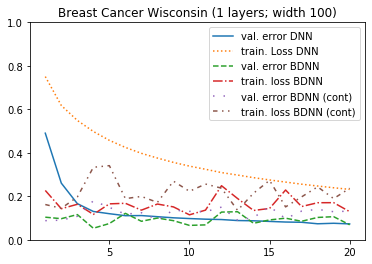}
\includegraphics[scale=0.4]{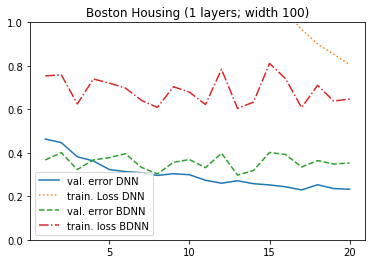}
\includegraphics[scale=0.4]{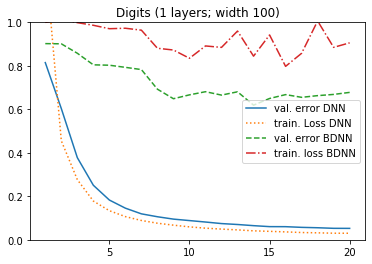}
\caption{Validation accuracy and training loss for each epoch for various datasets.} 
\label{fig:val_train_loss_datasets}
\vspace{-3mm}
\end{figure}

In Figure \ref{fig:val_train_loss_datasets} we show the average training and validation error (over $10$ random train-validation-test splits) in each epoch for the DS algorithm applied to BDNNs with integer weights and the Adam algorithm for DNNs. The results show that the validation and the training error for the BDNN are much more unstable and do not have a globally decreasing trend in contrast to the DNNs. Furthermore the validation error and the training error of the BDNN seem to be dependent, i.e. they increase or decrease at the same time. The same holds for BDNNs with continuous weights where both values are larger than for the integer version.

\begin{figure}[h!]
\centering
\includegraphics[scale=0.4,valign=t]{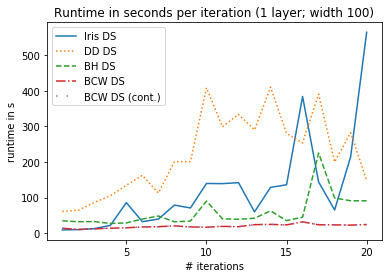}
\includegraphics[scale=0.4,valign=t]{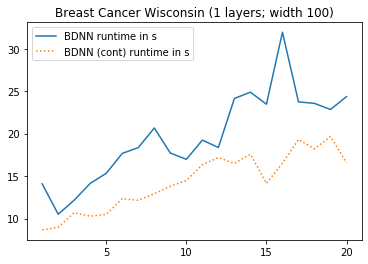}
\caption{Time to solve \eqref{eq:MINLP_reducedUVariables} in each epoch of the DS algorithm for various datasets.} 
\label{fig:runtime_datasets}
\end{figure}

In Figure \ref{fig:runtime_datasets} we show the average calculation time of Problem \ref{eq:MINLP_reducedUVariables} for each epoch of the DS algorithm for BDNNs with integer weights. Since in each iteration we add more integer variables to the Problem an increase in computation time would be the expectation. However although we can observe an increasing trend the runtime depends on the performance of the integer programming method of Gurobi and can fluctuate. The runtime of the continuous version on the BCW dataset seems to be much more stable with a smaller slope. 

\subsection{Robust Binarized Neural Networks}
In this section we test the robust BDNN model (RO-BDNN) presented in Section \ref{sec:adversarialAttacks}. To defend our model against possible attacks we consider uncertainty sets $U=U'\times \ldots \times U'$ where the set $U'$ is given by $U'=\{ \delta\in \R^n : \ \|\delta\|_0\le \varepsilon_d\}$ with different defense levels $\varepsilon_d\in \{ 0,0.25,0.5,0.75,1.0\}$. Since the impact of an attack depends on the size of the values of the different attributes we normalize the whole dataset by using the MinMaxScaler of the scikit-learn package and scale all attribute values to the interval $[0,1]$. Afterwards we train the RO-BDNN with uncertainty sets $U$ via DS algorithm with the same parameter setup as described in the previous section. After training we attack the test set by random $\|\cdot \|_0$-attacks of level $\varepsilon_a\in \{ 0, 0.5,1.0,1.5,2.0\}$. To this end for each test sample we draw a random attack vector $v_a\in\{-\varepsilon_a ,\varepsilon_a\}^n$ and add it to the test sample. As benchmark values we train classical DNNs on the training set and compare the performance on the attacked test sets. 

\begin{table}[h!]
\caption{Accuracy (in \%) of robust BDNN and DNN for attacked test sets.}
\label{tbl:robustAccuracy}
\begin{center}
\begin{tabular}{cc|l|llllll}
&&DNN&\multicolumn{6}{c}{BDNN} \\
Dataset & $\varepsilon_{a}$\textbackslash$\varepsilon_d$ & $0$ & $0$ & $0.025$ & $0.05$ & $0.1$ & $0.2$ & $0.3$ \\
\hline
BCW&0.0&96.1&95.8&95.9&\textbf{96.6}&95.9&85.3&67.7 \\
&0.1&95.7&95.3&95.7&95.7&\textbf{95.8}&84.1&67.7\\
&0.2&93.3&93.4&93.1&93.2&\textbf{93.9}&84.0&67.7\\
&0.5&80.5&\textbf{83.3}&76.5&79.3&76.7&74.4&66.7\\
&1.0&69.6&\textbf{70.9}&68.3&69.9&66.9&64.7&67.0\\
\hline
Iris&0.0&66.5&\textbf{89.5}&83.9&62.4&56.3&30.3&28.2 \\
&0.1&61.1&\textbf{82.4}&78.9&60.8&55.8&30.3&28.2\\
&0.2&55.1&\textbf{70.3}&69.5&58.7&51.8&30.3&28.2\\
&0.5&44.6&\textbf{49.2}&46.3&44.7&41.3&30.5&28.9\\
&1.0&38.5&\textbf{41.1}&32.9&35.5&36.6&30.5&28.7 \\
\hline
BH&0.0&\textbf{77.9}&77.2&73.2&70.5&63.4&51.1&51.1 \\
&0.1&\textbf{76.7}&73.9&71.1&70.6&63.4&51.2&51.4\\
&0.2&\textbf{73.9}&69.3&68.3&69.3&62.6&50.8&51.1\\
&0.5&\textbf{67.1}&57.8&63.7&60.5&59.8&47.8&48.8\\
&1.0&\textbf{60.5}&52.8&55.5&55.3&55.3&48.4&49.8
\end{tabular}
\end{center}
\vspace{-3mm}
\end{table}

In Table \ref{tbl:robustAccuracy} we show the average accuracies over $10$ random train-test-splits on the attacked test sets for different attack levels $\varepsilon_a$ and for different defense levels $\varepsilon_d$ for three datasets and networks with one hidden-layer of width $100$. The results show that the classical trade-off between robustness and accuracy which is often observed in interpretable robust machine learning (see e.g. \cite{dobriban2020provable,javanmard2020precise}) cannot always be observed for the BDNN. More precisely the classical trade-off can be described by the effect that robustness in the model leads to a lower accuracy for small attacks while the accuracies for larger attacks are better than for the non-robust models. We can see that the accuracy for non-attacked data ($\varepsilon_a=0$) drops for the Iris dataset after adding robust defenses of any size $\varepsilon_d$. Here enforcing the robustness seems to deteriorate the accuracy for all attack levels compared to the non-robust BDNN indicating that the classical BDNN is already robust for this dataset. However carefully adjusting the defense level to smaller values can still improve the accuracy. For the BCW dataset choosing a defense value of $\varepsilon_d=0.05$ or $\varepsilon_d=0.1$ increases the accuracy for small attack levels while for large attack values the non-robust BDNN performs best. Nevertheless using a defense value of $\varepsilon_d=0.05$ leads to good performances for larger attacks as well. Regarding the BH dataset the classical DNN performs better than the BDNN which was already observed in Table \ref{tbl:splitAlgo_otherDatasets}. However the results show that the non-robust BDNN performs better than the defended versions for small attacks, while for larger attacks the robust BDNN with defense level $\varepsilon_d=0.025$ performs better. In summary the BDNN seems to be very sensitive when enforcing robustness, leading sometimes to better performances for small or large attacks while sometimes the non-robust version can be the most robust one. The results indicate that the defense level has to be chosen carefully including very small values.

\section{Conclusion}
We show that binary deep neural networks can be modeled by mixed-integer programming formulations, which can be solved to global optimality by classical integer programming solvers. Additionally, we present a heuristic algorithm to derive local optimal solutions, leading to a better accuracy on small networks but hardly no improvement in calculation time. To overcome this issue we show that, using an iterative data-splitting heuristic (DS), we can decrease the number of integer variables and therefore the computation time. The results indicate that the solutions perform very well for small network architectures while suffering in terms of accuracy for larger architectures. Nevertheless they often achieve the best accuracy over all considered network architectures, which comes along with a significantly larger calculation time. In summary the results indicate that the DS method is favorable if small memory consumption and low evaluation complexity is desired. Additionally we consider a robust variant of the BDNN which sometimes suffers in terms of accuracy for different attack levels while achieving better accuracies than the non-robust DNN or BDNN for some attack-defense combinations. Nevertheless in future work the defense level should be adjusted carefully maybe involving a well designed validation process.

The mixed-integer programming formulation is very adjustable to variations of the model and could give new insights into the understanding of deep neural networks. The impact of different regularization methods, e.g. sparsity constraints, should be investigated in future works. This also motivates further research into the computational scalability of the IP methods and the study of other tractable reformulations or algorithms regarding the training of BDNNs.

\paragraph{Acknowledgments.} BB has been supported by BMBF through the German Research Chair at AIMS, administered by the Humboldt Foundation.

\bibliography{icmlref}

\begin{thebibliography}{48}
\providecommand{\natexlab}[1]{#1}
\providecommand{\url}[1]{\texttt{#1}}
\expandafter\ifx\csname urlstyle\endcsname\relax
  \providecommand{\doi}[1]{doi: #1}\else
  \providecommand{\doi}{doi: \begingroup \urlstyle{rm}\Url}\fi

\bibitem[Anderson et~al.(2020)Anderson, Huchette, Ma, Tjandraatmadja, and
  Vielma]{anderson2020strong}
Ross Anderson, Joey Huchette, Will Ma, Christian Tjandraatmadja, and Juan~Pablo
  Vielma.
\newblock Strong mixed-integer programming formulations for trained neural
  networks.
\newblock \emph{Mathematical Programming}, pages 1--37, 2020.

\bibitem[Bah and Kurtz(2020)]{bah2020integer}
Bubacarr Bah and Jannis Kurtz.
\newblock An integer programming approach to deep neural networks with binary
  activation functions.
\newblock \emph{arXiv preprint arXiv:2007.03326}, 2020.

\bibitem[Barlett and Downs(1992)]{barlett1992using}
Peter~L Barlett and Tom Downs.
\newblock Using random weights to train multilayer networks of hard-limiting
  units.
\newblock \emph{IEEE Transactions on Neural Networks}, 3\penalty0 (2):\penalty0
  202--210, 1992.

\bibitem[Baum and Haussler(1989)]{baum1989size}
Eric~B Baum and David Haussler.
\newblock What size net gives valid generalization?
\newblock In \emph{Advances in neural information processing systems}, pages
  81--90, 1989.

\bibitem[Bengio et~al.(2013)Bengio, L{\'e}onard, and
  Courville]{bengio2013estimating}
Yoshua Bengio, Nicholas L{\'e}onard, and Aaron Courville.
\newblock Estimating or propagating gradients through stochastic neurons for
  conditional computation.
\newblock \emph{arXiv preprint arXiv:1308.3432}, 2013.

\bibitem[Bertsimas et~al.(2019)Bertsimas, Dunn, Pawlowski, and
  Zhuo]{bertsimas2019robust}
Dimitris Bertsimas, Jack Dunn, Colin Pawlowski, and Ying~Daisy Zhuo.
\newblock Robust classification.
\newblock \emph{INFORMS Journal on Optimization}, 1\penalty0 (1):\penalty0
  2--34, 2019.

\bibitem[Corwin et~al.(1994)Corwin, Logar, and Oldham]{corwin1994iterative}
Edward~M Corwin, Antonette~M Logar, and William~JB Oldham.
\newblock An iterative method for training multilayer networks with threshold
  functions.
\newblock \emph{IEEE Transactions on Neural Networks}, 5\penalty0 (3):\penalty0
  507--508, 1994.

\bibitem[Courbariaux et~al.(2015)Courbariaux, Bengio, and
  David]{courbariaux2015binaryconnect}
Matthieu Courbariaux, Yoshua Bengio, and Jean-Pierre David.
\newblock Binaryconnect: Training deep neural networks with binary weights
  during propagations.
\newblock In \emph{Advances in neural information processing systems}, pages
  3123--3131, 2015.

\bibitem[Dobriban et~al.(2020)Dobriban, Hassani, Hong, and
  Robey]{dobriban2020provable}
Edgar Dobriban, Hamed Hassani, David Hong, and Alexander Robey.
\newblock Provable tradeoffs in adversarially robust classification.
\newblock \emph{arXiv preprint arXiv:2006.05161}, 2020.

\bibitem[Dua and Graff(2017)]{Dua:2019}
Dheeru Dua and Casey Graff.
\newblock {UCI} machine learning repository, 2017.
\newblock URL \url{http://archive.ics.uci.edu/ml}.

\bibitem[Fischetti and Jo(2018)]{fischetti2018deep}
Matteo Fischetti and Jason Jo.
\newblock Deep neural networks and mixed integer linear optimization.
\newblock \emph{Constraints}, 23\penalty0 (3):\penalty0 296--309, 2018.

\bibitem[Goerigk and Kurtz(2020)]{goerigk2020data}
Marc Goerigk and Jannis Kurtz.
\newblock Data-driven robust optimization using unsupervised deep learning.
\newblock \emph{arXiv preprint arXiv:2011.09769}, 2020.

\bibitem[Goodfellow et~al.(2016)Goodfellow, Bengio, and
  Courville]{goodfellow2016deep}
Ian Goodfellow, Yoshua Bengio, and Aaron Courville.
\newblock \emph{Deep learning}.
\newblock MIT press, 2016.

\bibitem[Goodman and Zeng(1994)]{goodman1994learning}
Rodney~M Goodman and Zheng Zeng.
\newblock A learning algorithm for multi-layer perceptrons with hard-limiting
  threshold units.
\newblock In \emph{Proceedings of IEEE Workshop on Neural Networks for Signal
  Processing}, pages 219--228. IEEE, 1994.

\bibitem[Gray and Michel(1992)]{gray1992training}
Donald~L Gray and Anthony~N Michel.
\newblock A training algorithm for binary feedforward neural networks.
\newblock \emph{IEEE Transactions on Neural Networks}, 3\penalty0 (2):\penalty0
  176--194, 1992.

\bibitem[Hampson and Volper(1990)]{hampson1990representing}
Steven~E Hampson and Dennis~J Volper.
\newblock Representing and learning boolean functions of multivalued features.
\newblock \emph{IEEE transactions on systems, man, and cybernetics},
  20\penalty0 (1):\penalty0 67--80, 1990.

\bibitem[Hubara et~al.(2016)Hubara, Courbariaux, Soudry, El-Yaniv, and
  Bengio]{hubara2016binarized}
Itay Hubara, Matthieu Courbariaux, Daniel Soudry, Ran El-Yaniv, and Yoshua
  Bengio.
\newblock Binarized neural networks: Training neural networks with weights and
  activations constrained to + 1 or-1.
\newblock \emph{arXiv preprint arXiv:1602.02830}, 2016.

\bibitem[Icarte et~al.(2019)Icarte, Illanes, Castro, Cire, McIlraith, and
  Beck]{icarte2019training}
Rodrigo~Toro Icarte, Le{\'o}n Illanes, Margarita~P Castro, Andre~A Cire,
  Sheila~A McIlraith, and J~Christopher Beck.
\newblock Training binarized neural networks using {MIP} and {CP}.
\newblock In \emph{International Conference on Principles and Practice of
  Constraint Programming}, pages 401--417. Springer, 2019.

\bibitem[Javanmard et~al.(2020)Javanmard, Soltanolkotabi, and
  Hassani]{javanmard2020precise}
Adel Javanmard, Mahdi Soltanolkotabi, and Hamed Hassani.
\newblock Precise tradeoffs in adversarial training for linear regression.
\newblock In \emph{Conference on Learning Theory}, pages 2034--2078. PMLR,
  2020.

\bibitem[Jia and Rinard(2020)]{jia2020efficient}
Kai Jia and Martin Rinard.
\newblock Efficient exact verification of binarized neural networks.
\newblock \emph{arXiv preprint arXiv:2005.03597}, 2020.

\bibitem[Khalil et~al.(2018)Khalil, Gupta, and
  Dilkina]{khalil2018combinatorial}
Elias~B Khalil, Amrita Gupta, and Bistra Dilkina.
\newblock Combinatorial attacks on binarized neural networks.
\newblock \emph{arXiv preprint arXiv:1810.03538}, 2018.

\bibitem[Kim and Smaragdis(2016)]{kim2016bitwise}
Minje Kim and Paris Smaragdis.
\newblock Bitwise neural networks.
\newblock \emph{arXiv preprint arXiv:1601.06071}, 2016.

\bibitem[Kohut and Steinbach(2004)]{kohut2004boolean}
Roman Kohut and Bernd Steinbach.
\newblock Boolean neural networks.
\newblock \emph{Transactions on Systems}, 2:\penalty0 420--425, 2004.

\bibitem[Kurakin et~al.(2016)Kurakin, Goodfellow, and
  Bengio]{kurakin2016adversarial}
Alexey Kurakin, Ian Goodfellow, and Samy Bengio.
\newblock Adversarial machine learning at scale.
\newblock \emph{arXiv preprint arXiv:1611.01236}, 2016.

\bibitem[Lazarus and Kochenderfer(2021)]{lazarus2021mixed}
Christopher Lazarus and Mykel~J Kochenderfer.
\newblock A mixed integer programming approach for verifying properties of
  binarized neural networks.
\newblock 2021.

\bibitem[LeCun et~al.(2015)LeCun, Bengio, and Hinton]{lecun2015deep}
Yann LeCun, Yoshua Bengio, and Geoffrey Hinton.
\newblock Deep learning.
\newblock \emph{nature}, 521\penalty0 (7553):\penalty0 436--444, 2015.

\bibitem[Maass(1994)]{maass1994perspectives}
Wolfgang Maass.
\newblock Perspectives of current research about the complexity of learning on
  neural nets.
\newblock In \emph{Theoretical advances in neural computation and learning},
  pages 295--336. Springer, 1994.

\bibitem[McCulloch and Pitts(1943)]{mcculloch1943logical}
Warren~S McCulloch and Walter Pitts.
\newblock A logical calculus of the ideas immanent in nervous activity.
\newblock \emph{The bulletin of mathematical biophysics}, 5\penalty0
  (4):\penalty0 115--133, 1943.

\bibitem[Montufar et~al.(2014)Montufar, Pascanu, Cho, and
  Bengio]{montufar2014number}
Guido~F Montufar, Razvan Pascanu, Kyunghyun Cho, and Yoshua Bengio.
\newblock On the number of linear regions of deep neural networks.
\newblock In \emph{Advances in Neural Information Processing Systems}, pages
  2924--2932, 2014.

\bibitem[Nahapetyan(2009)]{nahapetyan2009bilinear}
Artyom~G. Nahapetyan.
\newblock \emph{Bilinear Programming}, pages 279--282.
\newblock Springer US, Boston, MA, 2009.
\newblock ISBN 978-0-387-74759-0.
\newblock \doi{10.1007/978-0-387-74759-0_48}.
\newblock URL \url{https://doi.org/10.1007/978-0-387-74759-0_48}.

\bibitem[Plagianakos et~al.(2001)Plagianakos, Magoulas, Nousis, and
  Vrahatis]{plagianakos2001training}
VP~Plagianakos, GD~Magoulas, NK~Nousis, and MN~Vrahatis.
\newblock Training multilayer networks with discrete activation functions.
\newblock In \emph{IJCNN'01. International Joint Conference on Neural Networks.
  Proceedings (Cat. No. 01CH37222)}, volume~4, pages 2805--2810. IEEE, 2001.

\bibitem[Qin et~al.(2020)Qin, Gong, Liu, Bai, Song, and Sebe]{qin2020binary}
Haotong Qin, Ruihao Gong, Xianglong Liu, Xiao Bai, Jingkuan Song, and Nicu
  Sebe.
\newblock Binary neural networks: A survey.
\newblock \emph{Pattern Recognition}, page 107281, 2020.

\bibitem[Raghu et~al.(2017)Raghu, Poole, Kleinberg, Ganguli, and
  Sohl-Dickstein]{raghu2017expressive}
Maithra Raghu, Ben Poole, Jon Kleinberg, Surya Ganguli, and Jascha
  Sohl-Dickstein.
\newblock On the expressive power of deep neural networks.
\newblock In \emph{International Conference on Machine Learning}, pages
  2847--2854. PMLR, 2017.

\bibitem[Rastegari et~al.(2016)Rastegari, Ordonez, Redmon, and
  Farhadi]{rastegari2016xnor}
Mohammad Rastegari, Vicente Ordonez, Joseph Redmon, and Ali Farhadi.
\newblock {XNOR-N}et: Imagenet classification using binary convolutional neural
  networks.
\newblock In \emph{European conference on computer vision}, pages 525--542.
  Springer, 2016.

\bibitem[Rister and Rubin(2017)]{rister2017piecewise}
Blaine Rister and Daniel~L Rubin.
\newblock Piecewise convexity of artificial neural networks.
\newblock \emph{Neural Networks}, 94:\penalty0 34--45, 2017.

\bibitem[Rumelhart et~al.(1986)Rumelhart, Hinton, and
  Williams]{rumelhart1986learning}
David~E Rumelhart, Geoffrey~E Hinton, and Ronald~J Williams.
\newblock Learning representations by back-propagating errors.
\newblock \emph{nature}, 323\penalty0 (6088):\penalty0 533--536, 1986.

\bibitem[Sakurai(1993)]{sakurai1993tighter}
Akito Sakurai.
\newblock Tighter bounds of the {VC}-dimension of three layer networks.
\newblock In \emph{Proceedings of the World Congress on Neural Networks},
  volume~3, pages 540--543. Erlbaum, 1993.

\bibitem[Thorbjarnarson and Yorke-Smith(2020)]{thorbjarnarson2020training}
T{\'o}mas Thorbjarnarson and Neil Yorke-Smith.
\newblock On training neural networks with mixed integer programming.
\newblock \emph{arXiv preprint arXiv:2009.03825}, 2020.

\bibitem[Tjeng et~al.(2017)Tjeng, Xiao, and Tedrake]{tjeng2017evaluating}
Vincent Tjeng, Kai Xiao, and Russ Tedrake.
\newblock Evaluating robustness of neural networks with mixed integer
  programming.
\newblock \emph{arXiv preprint arXiv:1711.07356}, 2017.

\bibitem[Toms(1990)]{toms1990training}
DJ~Toms.
\newblock Training binary node feedforward neural networks by back propagation
  of error.
\newblock \emph{Electronics letters}, 26\penalty0 (21):\penalty0 1745--1746,
  1990.

\bibitem[Venzke and Chatzivasileiadis(2020)]{venzke2020verification}
Andreas Venzke and Spyros Chatzivasileiadis.
\newblock Verification of neural network behaviour: Formal guarantees for power
  system applications.
\newblock \emph{IEEE Transactions on Smart Grid}, 12\penalty0 (1):\penalty0
  383--397, 2020.

\bibitem[Venzke et~al.(2020)Venzke, Qu, Low, and
  Chatzivasileiadis]{venzkeRobustDNN}
Andreas Venzke, Guannan Qu, Steven Low, and Spyros Chatzivasileiadis.
\newblock Learning optimal power flow: Worst-case guarantees for neural
  networks.
\newblock In \emph{2020 IEEE International Conference on Communications,
  Control, and Computing Technologies for Smart Grids (SmartGridComm)}, pages
  1--7, 2020.
\newblock \doi{10.1109/SmartGridComm47815.2020.9302963}.

\bibitem[Wang et~al.(2018)Wang, Balestriero, and Baraniuk]{wang2018max}
Zichao Wang, Randall Balestriero, and Richard Baraniuk.
\newblock A max-affine spline perspective of recurrent neural networks.
\newblock In \emph{International Conference on Learning Representations}, 2018.

\bibitem[Widrow and Winter(1988)]{widrow1988neural}
Bernard Widrow and Rodney Winter.
\newblock Neural nets for adaptive filtering and adaptive pattern recognition.
\newblock \emph{Computer}, 21\penalty0 (3):\penalty0 25--39, 1988.

\bibitem[Wolsey(1998)]{wolsey1998integer}
Laurence~A Wolsey.
\newblock \emph{Integer programming}, volume~52.
\newblock John Wiley \& Sons, 1998.

\bibitem[Xu et~al.(2009{\natexlab{a}})Xu, Caramanis, and Mannor]{xu2009robust}
Huan Xu, Constantine Caramanis, and Shie Mannor.
\newblock Robust regression and {LASSO}.
\newblock In \emph{Advances in Neural Information Processing Systems}, pages
  1801--1808, 2009{\natexlab{a}}.

\bibitem[Xu et~al.(2009{\natexlab{b}})Xu, Caramanis, and
  Mannor]{xu2009robustness}
Huan Xu, Constantine Caramanis, and Shie Mannor.
\newblock Robustness and regularization of support vector machines.
\newblock \emph{Journal of machine learning research}, 10\penalty0
  (Jul):\penalty0 1485--1510, 2009{\natexlab{b}}.

\bibitem[Yuan et~al.(2019)Yuan, He, Zhu, and Li]{yuan2019adversarial}
Xiaoyong Yuan, Pan He, Qile Zhu, and Xiaolin Li.
\newblock Adversarial examples: Attacks and defenses for deep learning.
\newblock \emph{IEEE transactions on neural networks and learning systems},
  30\penalty0 (9):\penalty0 2805--2824, 2019.

\end{thebibliography}

\end{document}